\documentclass[11pt]{article}

\usepackage{graphicx}
\usepackage{url}
\usepackage{amsmath,amsfonts,amssymb,amsthm,mathtools,euscript,mathrsfs}
\usepackage{verbatim}
\usepackage{color}
\usepackage{float}
\usepackage[usenames,dvipsnames]{xcolor}
\usepackage{mdwlist,multicol}
\usepackage{caption,subcaption}
\usepackage{setspace}
\usepackage{bbm}
\usepackage{epstopdf}
\usepackage{cite}
\usepackage[shortlabels]{enumitem}
\usepackage{authblk}

\usepackage[top=1in, bottom=1in, left=1in, right=1in]{geometry}

\usepackage{hyperref}
\hypersetup{
    colorlinks=true,
    linkcolor=black,
    filecolor=magenta,      
    urlcolor=blue,
}

\newcommand{\rmd}{\mathrm{d}}           % derivatives
\newcommand{\N}{\mathbb{N}}

\newcommand{\R}{\mathbb{R}}

\newcommand{\K}{\mathcal{K}}

\renewcommand{\L}{\mathscr{L}}

\newcommand{\ep}{\varepsilon}

\newcommand{\tip}{\textrm{tip}}
\newcommand{\erf}{\textrm{erf}}

\newcommand{\wt}{\widetilde}
\newcommand{\wh}{\widehat}

\renewcommand{\max}{{\textrm{max}}}

	\theoremstyle{plain}

	  \newtheorem{assumption}{Assumption}
        \newtheorem{hypothesis}{Hypothesis}
	\newtheorem{proposition}{Proposition} 
	\newtheorem{theorem}{Theorem}
	
	\theoremstyle{definition}
        \newtheorem{remark}{Remark}

\title{Kernel-dependent pattern formation in a population model with nonlocal facilitation and competition}

\author[1]{Olivia Clifton}
\author[2]{Stephanie Dodson}
\author[1,3]{Daniel B. Cooney}
\affil[1]{Department of Mathematics, University of Illinois Urbana-Champaign, Urbana, IL, USA}
\affil[2,*]{Mathematics Department, Colby College, Waterville, ME, USA}
\affil[3]{Carl R. Woese Institute for Genomic Biology, University of Illinois Urbana-Champaign, Urbana, IL, USA}
\affil[*]{Correspondence to sdodson@colby.edu}

\begin{document}

\maketitle

\begin{abstract}
Spatial patterns, such as those in dryland vegetation models,  have historically been studied in systems of reaction-diffusion systems with pattern onset via a Turing bifurcation from a spatially uniform state. More recently, spatial patterns have been considered in models that incorporate spatially extended interactions via nonlocal interaction kernels. It remains largely underexplored if and how the choice of nonlocal interaction kernel contributes to differences in pattern formation and persistence, particularly in models that contain competition and facilitation. Here, we investigate spatial patterns in a  reaction-diffusion model for a single species that includes nonlocal competition and facilitation processes; Gaussian, exponential, algebraic, hat, and smooth hat kernels are considered as specific examples. Via a center manifold analysis, and using the relative spatial scale of competition to facilitation and the death rate as bifurcation parameters, we identify that the choice of kernel has impacts on the pattern forming bifurcation. Bifurcations using the Gaussian, exponential, and algebraic kernels largely follow expectations of Turing patterns, but patterns in the hat and smooth hat kernels can form even when the scale of competition is less than that of facilitation. The dynamics of patterns far from onset are investigated via numerical continuation methods. The model produces the so-called ``Turing-before-Tipping'' phenomenon demonstrating that the arrangement into spatial patterns is an effective resilience mechanism against harsh conditions. Again, there is a kernel-dependent dichotomy in pattern behavior. Early warning signs for population extinction are observed with the Gaussian, exponential, and algebraic kernels, but not under the hat or smooth hat cases.

\end{abstract}


35Q92: PDEs in connection with biology, chemistry and other natural sciences\\
5B36: Pattern formations in context of PDEs\\
35K57: Reaction-diffusion equations\\
35R09: Integro-partial differential equations\\

{\bf Acknowledgments:} This research collaboration arose from a project at the 2023 Mathematics Research Community on Complex Social Systems, which was organized by the American Mathematical Society and supported by the US National Science Foundation through grant DMS-1916439. We thank Jeungeun Park, Rebecca Hardenbrook, and Jo\~{a}o Souto-Maior for their helpful discussions and collaboration at the MRC workshop on an earlier version of this project.

\section{Introduction}

Spatial patterns are abundant across a range of biological, chemical, and environmental applications. In particular, spatial patterning has recently deepened understanding of vegetation dynamics in semi-arid ecosystems. %
In spatially uniform models of vegetation dynamics, increasing environmental stress often leads to a catastrophic transition—or “tipping”—from a vegetated to a barren state. However, spatially patterned vegetation can delay or even prevent such collapse, enabling ecosystems to persist under harsher conditions than uniform models would predict \cite{rietkerk2021evasion}. The formation of such patterns is now recognized both as a warning sign for desertification and as a mechanism for ecosystem resilience \cite{rietkerk2004self,rietkerk2021evasion}. Therefore, studying how such patterns emerge and persist can provide important insight for the application. %
 
In the vegetation patterns literature, the modulational (Turing) instability of a uniform vegetated state is commonly triggered by increasing environmental stress or worsening environmental conditions. The Turing bifurcation is a classical mechanism, in which patterns emerge from a spatially uniform state due to competing forces of short range activation and long-range inhibition \cite{turing1952chemical,gierer1972theory,tyson2026revisiting}. Once established, the persistence and stability of spatially periodic patterns can be understood in terms of a \emph{Busse balloon}, which defines the wavenumbers of stable patterned states. Recent studies have focused on characterizing Busse balloons for semi-arid ecosystems \cite{ferre2025vegetation,vanderStelt2013riseandfall,vanderveken2023existence}, the mechanisms by which patterns lose stability \cite{doelman2012hopf,sewalt2017spatially,vanderStelt2013riseandfall}, and how periodic states evolve through the Busse ballon region under time-varying conditions \cite{asch2025slow,siteur2014beyond,hamster2025blurring}. Observations show that characteristic wavelengths tend to increase with environmental harshness, with a rich and structured transition sequence near collapse that may provide a warning sign \cite{gilad2004ecosystem,gowda2016assessing,rietkerk2002self,rietkerk2004self,vonHardenberg2001diversity,meron2018patterns}.

Mathematically, much work on dryland systems has centered on reaction-diffusion models such as Klausmeier--Gray--Scott-type systems \cite{bastiaansen2018multistability,lejeune2004vegetation,sewalt2017spatially,siero2015striped,vanderStelt2013riseandfall}, and various extensions such as models with pulsed precipitation \cite{eigentler2020effects,gandhi2023pulsed,gandhi2025flow}. In particular, detailed analyses of vegetation Busse balloons and pattern multistability have primarily been done in these types of models \cite{ferre2025vegetation,hamster2025blurring,siteur2014beyond,vanderStelt2013riseandfall}. In the reaction-diffusion framework, differences in the diffusivities of biomass and water generate the necessary separation of spatial scales that underlies pattern formation. An alternative modeling approach in the mathematical biology literature, however, includes the separation of scales explicitly through nonlocal interactions, where intra- and inter-species facilitation and competition are represented by spatial kernels \cite{lefever1997origin,ruiz2020patterns,ruiz2017fairy,tega2022spatio,tlidi2008vegetation,patterson2024pattern}. These nonlocal models can capture not only the effective scale separation of reaction-diffusion systems but also mechanistic spatial processes such as root competition and facilitative processes such as canopy shading, shielding from fire, the sharing of resources, and benefits of aggregation. 
Nevertheless, such models have received comparably little attention in the mathematical literature, and warrant a systematic study.

It is well-known that scalar nonlocal equations exhibit a much richer set of solutions than their local counterparts \cite{Britton1989, Britton1990,Gourley2001,Kondo2017}. The addition of nonlocal terms adds the possibility of non-monotonic solutions, including pulses, periodic waves, and patterns. Additionally, the ability to choose the nonlocal kernel(s) adds a significant degree of freedom to such equations, and the impact of that choice is not always clear. In many applications, the specific interaction kernel used is chosen for mathematically convenient properties (such as Gaussian or exponential kernels).   However, the choice of interaction or dispersal kernel can have important consequences for patterns. A study of nonlocal intraspecies competition in a two-variable predator-prey system demonstrated that patterns are more likely to form with kernel structures that exhibit a broad peak and narrow tails \cite{Merchant2011}. Similarly, nonlocal dispersal governed by Gaussian and exponential kernels versus a top-hat kernel yielded major qualitative differences in pattern onset in a $\lambda-\omega$ system \cite{Sherratt2014}. Conversely, an analysis of nonlocal seed dispersal in the Klausmeier model found that the kernel width and dispersal rates proved to be more impactful on pattern formation than the type of kernel \cite{Bennett2019}.

 In this manuscript, we study a simple scalar population model that includes effects of both nonlocal competition and nonlocal facilitation. The uniform dynamics in the model exhibit tipping, and thus we study pattern formation and persistence past the tipping point. 
 Throughout, we are interested in the impact of kernel choice on bifurcations, both qualitatively and quantitatively.

Specifically, we study the following model for the formation of patterns in one spatial dimension:

\begin{equation}\label{e:nonlocal_veg}
    \dfrac{\partial u}{\partial t} = gu \left(1-c+ \frac{f_\textrm{max}   \K_f * u}{f_\textrm{max}/f_0 +  \K_{f} * u}\right)\left( 1-\dfrac{\K_{c}*u}{M}\right) - du + k_d \dfrac{\partial^{2} u}{\partial x^2}.
\end{equation}

This particular model was formulated by Hermsen \cite{hermsen2022emergent} to investigate the evolution of altruistic colonies and emergent dynamics that underlie multilevel selection. 
It is notably reminiscent of the vegetation model considered by Lefevre and Lejeune in \cite{lefever1997origin}. 
Briefly, $u=u(x,t)$ defines the population density and kernels $\mathcal{K}_{f,c} = \mathcal{K}_{f,c}(x; \sigma_{f,c})$ represent the nonlocal facilitation and competition, respectively, where parameters $\sigma_{f,c}$ control the spatial length scales of the two processes. Our main bifurcation parameters are $d$, the linear death rate, standing in for environmental harshness, and $\sigma_c$, the spatial length scale of competition. Throughout, we fix the spatial length scale of facilitation at $\sigma_f = 1$, so that using $\sigma_c$ as a bifurcation parameter is equivalent to considering the ratio $\frac{\sigma_c}{\sigma_f}$ of the two spatial scales.

 As techniques, we employ a center manifold expansion in periodic function spaces to show the existence of a pattern forming bifurcation, and examine coperiodic stability of patterns by finding kernel-dependent coefficients which determine if bifurcations are super- or subcritical. The latter is of interest for two reasons: first, subcritical bifurcations may suggest the existence of localized structures and/or homoclinic snaking. Second, recent work by Staal and Doelman finds that in the limit where the onset of patterns is very close to the tipping point, whether bifurcations are super- or subcritical can determine whether patterns outlast the tipping point \cite{staal2025evasion}. It is therefore of interest to be able to determine sub- or supercriticality for a wide range of parameters and kernels without doing lengthy numerical continuation in each case. 
 %We implicitly assume a degree of generality in the dynamics and leave the kernels general in the analysis of pattern onset. 
 We also use numerical continuation methods to illustrate the regions of pattern stability far from initial onset. This is complemented by some direct simulations investigating pattern reselection
 %how patterns transition through this region of pattern multistability 
 as the environmental harshness parameter slowly increases. We thus show that the phenomenon of ``Turing-before-tipping'' can be found in a simple scalar nonlocal model, and we give examples of how pattern resilience depends on the choice of nonlocal kernel.

Interestingly, our results find a dichotomy of pattern behavior between two kernel classes. The first class of kernels  can be defined as those whose Fourier transforms are nonnegative, which we refer to as ``Gaussian-type." The second class are those whose Fourier transforms take on negative values, and we refer to them as ``hat-type," referencing the example of the top-hat function. 

For Gaussian-type kernels, we find robust evidence of the typical "Turing-before-tipping" story. Under mild assumptions, increasing the death rate $d$ generically induces a pattern-forming instability. 
Past onset, numerical continuation then reveals regions of stable patterned states that persist past the tipping point as $d$ is increased. Direct simulations with slowly increasing $d$ display the typical pattern of wavenumber reselection towards smaller wavenumbers, possibly through a period-doubling and/or Hopf instability. We note that despite the qualitatively similar behavior, kernel choice can nevertheless significantly affect the bifurcation points as well as pattern wavelength.

For hat-type kernels, the story is less straightforward. First, we find that the death rate is not generically an organizing parameter for modulational instability for these kernels. Instead, typically, the uniform state is either stable or unstable to spatially periodic perturbations for all values of $d \ge 0$. Although it is possible to find kernels and parameters for which increasing $d$ would generate a Turing instability, it takes careful tuning and requires either unphysically large diffusion or small $\sigma_c$. The scale of competition $\sigma_c$ does always induce a modulational instability, but for small diffusion, patterns are present for $\sigma_c <1$, in contradiction to the standard notion of a Turing mechanism. Numerical continuation again provides evidence of pattern resilience, but, interestingly, hat-type kernels did not produce the typical early warning sign of wavelengths progressively becoming unstable as $d$ increases; instead, patterns typically remain stable, without changing wavenumber, until they abruptly disappear. 
In sum, our results investigate the "Turing-before-tipping" phenomenon in a scalar nonlocal model quite different to the reaction-diffusion and pulsed-precipitation frameworks. We find that the choice of nonlocal interaction kernels can affect all aspects of these bifurcations, from pattern onset to pattern reselection in the Busse balloon. 
%

%{\color{blue}[Paper Outline.]} 
The manuscript is outlined as follows. A full model description and mathematical assumptions on kernels can be found in Section~\ref{s:setup}. Section~\ref{s:uniform} defines the spatially uniform states and relevant linearizations. Section~\ref{s:CM} presents the main results about pattern existence and the bifurcation of patterns from the positive uniform state. Finally, the dynamics of patterns far from onset is discussed in Section~\ref{s:far_from_onset}. We end with a discussion of implications and further directions in Section \ref{s:discuss}.

\section{Mathematical Background \& Model Description}\label{s:setup}

We consider the equation 
\begin{equation}\label{e:main}
    \dfrac{\partial u}{\partial t} = g u \left(1-c + \frac{f_\textrm{max}  \K_f * u}{f_\textrm{max}/f_0 +  \K_{f} * u}\right)\left( 1-\dfrac{\K_{c}*u}{M}\right) - du + k_d \dfrac{\partial^{2} u}{\partial x^2}, \qquad x \in \R, \textrm{}
\end{equation}
 where $g$ is a baseline population growth rate, $c$ is the cost of facilitation to the individual, $M$ sets a carrying capacity above which competition will directly decrease population size, and $d$ is the linear death rate. The nonlocal facilitation and competition are described respectively by a facilitation kernel $\K_f$ and a competition kernel $\K_c$. The benefit of facilitation has a saturating dependence on the weighted interactions $\K_f * u$ -- when $u$ is small, the facilitation is approximately $f_0 \K_f *u$, but in the limit of many facilitative interactions $\K_f * u$, the benefit approaches a maximum $f_{\textrm{max}}$. Furthermore, the diffusion term with coefficient $k_d$ captures spatial motility of the population. 

\paragraph{Parameters.} In addition to the non-negative parameters $g,c,k_d, f_0, f_\max,M$, and $d$, we consider two spatial scaling parameters $\sigma_f, \sigma_c$ which control $\K_f,\K_c$ via $\K_{c,f}(x) = \frac{1}{\sigma}\wt{\K}_{c,f}(\frac{x}{\sigma})$. We fix $\sigma_f = 1$, so that the spatial scale of competition $\sigma_c >0$ represents the ratio $\frac{\sigma_c}{\sigma_f}$.

\begin{assumption}[Assumptions on Convolution Kernels]\label{a:kernels}

We assume the following for the nonlocal competition and facilitation kernels $\K_c, \K_f:$
\begin{enumerate}
    \item There exist functions ${\wt\K}_c, {\wt\K}_f$  such that $\K_{c}(x) = \frac{1}{\sigma_c}\wt{\K}_{c}(\frac{x}{\sigma_c})$, $\K_{f}(x) = \frac{1}{\sigma_f}\wt{\K}_{f}(\frac{x}{\sigma_f})$
    \item $\wt\K_c, \wt\K_f \in C^1(\R, \R)$, with $\int_{-\infty}^\infty \K_c(x)dx = \int_{-\infty}^\infty \K_{f}(x)dx = \int_{-\infty}^\infty \wt\K_c(x)dx = \int_{-\infty}^\infty \wt\K_f(x)dx = 1$;
    \item $\wt\K_c, \wt\K_f$ are even (radial) functions.
    \item $\int_{-\infty}^\infty x^2\wt\K_c(x)dx = x^2\int_{-\infty}^\infty \wt\K_f(x)dx = 1$, when these integrals are defined. If one (or both) is not, instead assume $\int_{-1}^1\wt\K_{c,f}(x)dx = \frac{2}{3}$. 
    %\item $\K_c, \K_f$ are monotone decreasing in $|x|$. 
\end{enumerate}
\end{assumption}

\subsection{Examples of Nonlocal Interaction Kernels}\label{ss:kernels}

The choice of interaction kernels $\K_c, \K_f$  in equation~\eqref{e:main} leads to a greater degree of freedom than a similar local equation would have. 
 We phrase results in section \ref{s:CM} in general terms, but to illustrate practical effects of the kernel choice, we will focus on a number of examples throughout the paper. In those examples, we assume that the facilitative and competitive kernels are the same, differing only by the spatial scaling parameters $\sigma_f, \sigma_c$. 
 
 We distinguish between two classes of kernels depending on their Fourier transforms. The first class, which we refer to as Gaussian-type, are those whose Fourier transforms are nonnegative. Many examples of these kernels model situations where interactions can occur at arbitrarily large distance and the interaction strength decreases with distance. The three example kernels (Gaussian, exponential, and algebraic) display differing rates of decrease. For the Gaussian kernel, competition drops off with distance at a faster-than-exponential rate; for the algebraic kernel, the dropoff is much slower, decaying like $\frac{1}{x^2}$. Long-range competitive interactions can arise physically when considering effects like competition for groundwater during a rainstorm, where the ground becomes saturated and surface water can diffuse over long distances. Not all kernels of Gaussian type have arbitrarily long-range interactions, however -- for instance, the triangle kernel $\K(x) = \max\left(\frac{1}{\ell} (x - \frac{1}{\ell}), 0 \right)$ from \cite{vanderVoort2026vegetation} has a finite range of interactions, but nonnegative Fourier transform. 

 The second class of kernels, which we refer to as hat-type, are those whose Fourier transforms have negative values. Many examples of these kernels come from interactions which drop off sharply and occur only within a finite radius. Such interactions may arise, for instance, through root systems.

We present below examples of Gaussian-type kernels (Gaussian, exponential, and algebraic) and the hat-like kernels (top hat and smooth hat) that will be used in numerical simulations throughout the paper. We also provide illustration of each of these kernels and their Fourier transform in Figure \ref{fig:kernels}. 

\begin{figure}[!ht]
    \centering
    \includegraphics[width=0.4\linewidth]{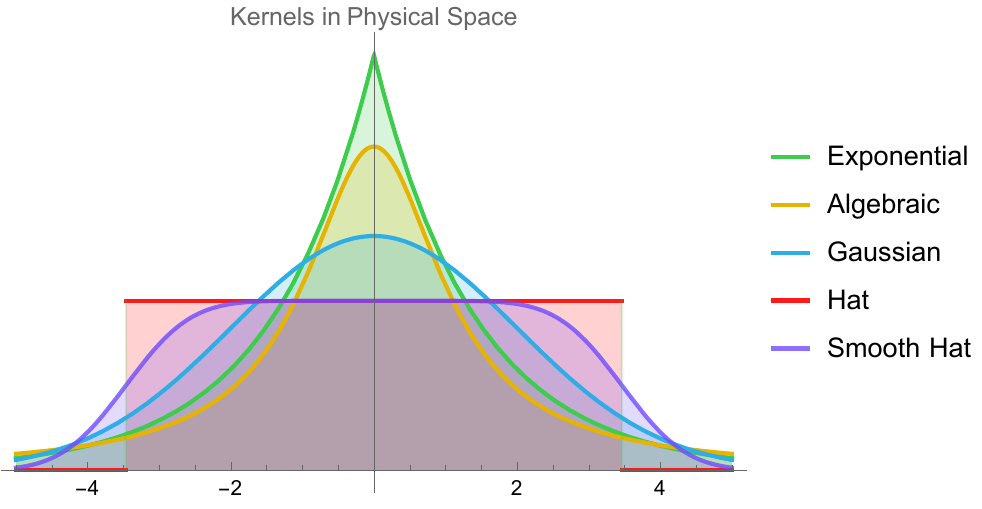}
    \includegraphics[width=0.4\linewidth]{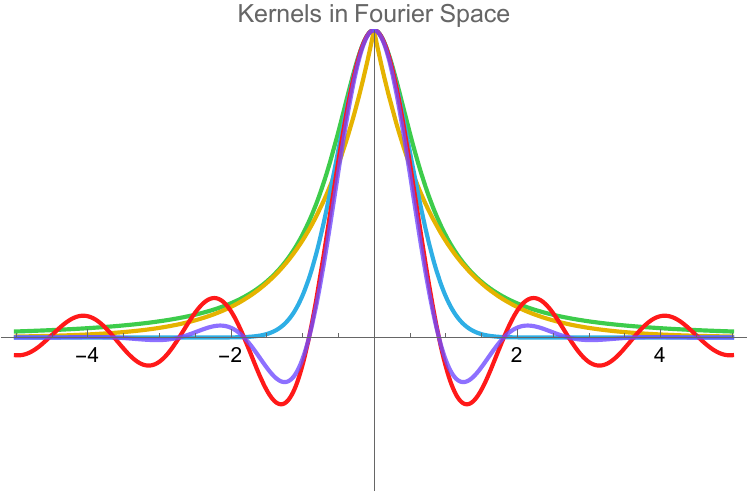}
    \caption{Plots of example kernels with $\sigma_c = 2$. Left, plots of kernels in physical space. Right, plots of kernels in Fourier space. }
    \label{fig:kernels}
\end{figure}

\bf Examples of Gaussian-type kernels: \rm 
\begin{itemize}
    \item Gaussian: $\K(x) = \frac{1}{\sigma\sqrt{2\pi}}e^{-\frac{x^2}{2\sigma^2}}$
    \item Exponential: $\K(x) = \frac{1}{\sigma\sqrt{2}}e^{-\frac{|x|\sqrt{2}}{\sigma}}$
    \item Algebraic: $\K(x) = \frac{\sigma}{\pi\sqrt{3}(\frac{\sigma^2}{3} + x^2)}$
\end{itemize}

\bf Examples of Hat-type kernels: \rm 
\begin{itemize}
    \item Hat (top-hat): $\K(x) = \frac{1}{2\sigma\sqrt{3}}\chi_{[-\sigma\sqrt{3},\sigma\sqrt{3}]}$
    \item Smooth Hat (bowler-hat): $\K(x) = \frac{1}{4 \sigma \sqrt{3}} (\erf([\sigma\sqrt{3}- x) + \erf(\sigma\sqrt{3} + x))$
\end{itemize}

\begin{remark}
    The careful observer will note that the exponential kernel and hat kernel do not satisfy the smoothness condition in Assumption \ref{a:kernels} above. However, in practice they can be smoothly approximated with an arbitrary degree of precision, so the difference is negligible for computational purposes, and we list the unsmoothed versions for simplicity. 
\end{remark}

\subsection{Normalization of Kernel Width $\sigma$.} 

A main goal of comparing behavior across kernels is to test the implicit assumption that the spatial length scale $\sigma$ of interactions is more important than the shape of the kernel. %We will compare bifurcations for kernels with different shapes but comparable spatial scalings. 
To do so, however, we must define what it means for two different kernels to have the same effective spatial scale. We normalize the spatial scale $\sigma$ of a kernel using the second moment of the kernel:
\[\sigma^2 := 
\int_\R x^2\wt\K(x)dx,
\]
when it is finite. In other words, $\sigma$ is the standard deviation of the kernel. The algebraic kernel, which does not have a well-defined variance, is instead normalized so that $2/3$ of the mass lies within $[-\sigma,\sigma]$, approximately the same as for the Gaussian. 

This choice of normalization can also be interpreted as fixing the value of $\wh{\widetilde{\K}}''(0)$, or, equivalently, the second nonzero Taylor coefficient of the kernels in Fourier space about $\omega = 0$. The value of $\wh{\widetilde{\K}}(0)$, the first Taylor coefficient, is already fixed to $1$ by the area condition $\int_\R\K(x)dx = 1$ (Assumption \ref{a:kernels}, part 2), so choosing $\sigma$ in this way means that all the kernels have the same Taylor expansion in Fourier space, up to order 2, when defined.

\section{Uniform Dynamics and Tipping} \label{s:uniform}

In order to understand the existence and resilience of patterns, we must first understand the dynamics of the system without spatial variation. First, seeking spatially uniform constant solutions $u \equiv u_0$ to \eqref{e:main}, we find three branches of uniform states:
\begin{align*}
   u_0 &= 0, \frac{M\left( (1-c-\frac{d}{g}) + f_\max  (1-\frac{(1 - c)}{f_0 M} ) \pm \sqrt{ \left((1-c-\frac{d}{g} + 
   f_\max (1-\frac{1-c}{f_0 M})\right)^2-\frac{4 f_\max (f_\max+1 - c)    (\frac{d}{g} + c-1)}{M f_0}}\right) }{2  (1 - c + f_\max) }.
\end{align*}
We denote by $u_*$ the (largest) nonnegative uniform state. Note that the value of  $u_*$ equals the carrying capacity $M$ when $d=0$, and decreases as $d$ increases. Furthermore, it can be shown that $u_*$ is stable to spatially homogeneous perturbations.

\paragraph{Tipping case.} If $f_0 > \frac{1-c}{M}$, then, $u_*$ is lost to a fold bifurcation at a critical value $d_\tip$ as depicted in Figure~\ref{fig:uniform_sols}. The existence of the fold, or tipping point, represents a sudden jump from a positive uniform steady-state to a state where only the zero solution is an equilibrium. At $d = d_\tip$, the critical value of $u_*$ at collapse is
\begin{equation}
       u_\tip := \frac{f_\max}{f_0}\left(\sqrt{\frac{f_\max + f_0 M}{f_\max + 1 - c}}-1\right) >0;
\end{equation}

note that this has a finite limit $u_{\tip}^{\lim} = \frac{M}{2} - \frac{1-c}{2f_0}$ as $f_{\max} \to \infty$. 

\paragraph{Non-tipping case.} If $f_0 < \frac{1-c}{M}$, the uniform bifurcation structure is a different case;
$u_*$ becomes negative after a transcritical bifurcation at $d = g(1-c)$. This is the non-tipping case, since 
 the uniform state connects continuously to 0 without a sudden collapse. We will focus primarily on the scenario $f_0 > \frac{1-c}{M}$, since it describes a case of bifurcation-induced tipping, where increasing $d$ can cause a sudden, irreversible jump between solution branches.

In Figure \ref{fig:uniform_sols}, we illustrate how the behavior of the spatially uniform equilibria depends on the death rate $d$. For the tipping case (Figure \ref{fig:uniform_sols}a), we see that the equilibrium $u_*$ is the only biologically feasible state that is stable for low death rate $d$. As $d$ increases, the trivial state and $u_*$ become bistable, and then the trivial state becomes the only stable equilibrium as $d$ increases past the tipping point $d_{\tip}$. For the non-tipping case (Figure \ref{fig:uniform_sols}b), we see that the equilibrium $u_*$ again the sole stable feasible state for sufficiently small $d$, and that this equilibrium is lost to a transcritical bifurcation as $d$ increases.

\begin{figure}
    \centering
    \includegraphics[width=0.42\linewidth]{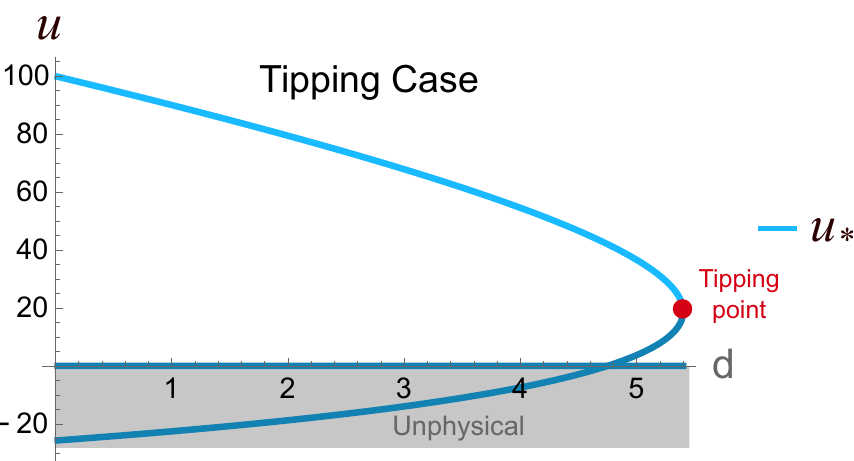}
     \includegraphics[width=0.42\linewidth]{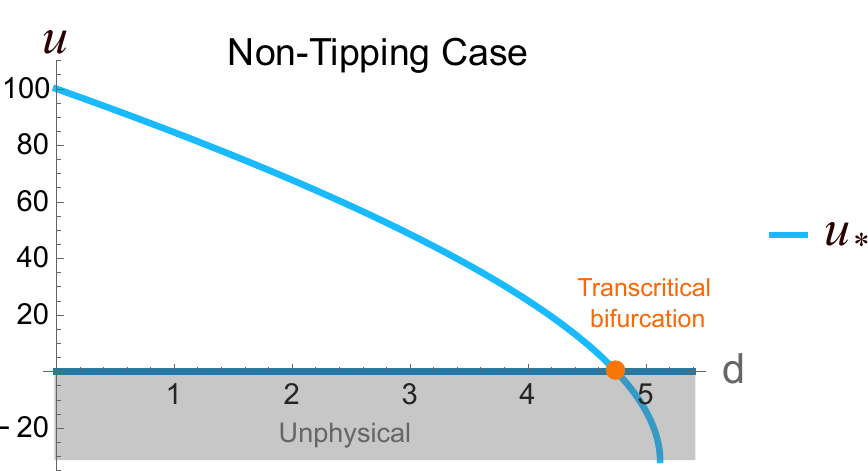}
    \caption{Plots of uniform solutions $u$ as the death rate $d$ varies (negative solutions are unphysical), in each of 2 cases. (a) Example of tipping case, with parameter values from Figure \ref{fig:par}. (b) Example of non-tipping case, with $f_0 = 0.005$ and all other parameters the same as in panel (a). }
    \label{fig:uniform_sols}
\end{figure}

\subsection{Baseline Simulation Parameters}

In Figure \ref{fig:par}, we present the baseline parameters that will be used for the numerical simulation of our PDE model. The reaction parameters are chosen so that the spatially uniform dynamics are in the tipping regime, and we set our default spatial scales of facilitation $\sigma_c = 1$ and competition $\sigma_c = 4$ to correspond to the Turing mechanism of short-range activation and long-range inhibition.

\begin{figure}[!ht]
\centering
\begin{tabular}{|c | c | c | 
}
 \hline
 Parameter &  Value used in numerics & Definition \\ [0.5ex] 
 \hline\hline
 $\sigma_f$ & 1  &  spatial scale of facilitation\\ 
 \hline
 $\sigma_c$ & 4  &  spatial scale of competition \\ 
 \hline
 $g$ & 5 &  basal growth rate\\ 
 \hline
  $c$ & 0.05 & cost of facilitation to individual \\ 
 \hline
 $M$ & 100  & carrying capacity\\
 \hline
 $k_d$ & 0.03 & coefficient of diffusion (dispersal)\\
 \hline
 $f_0$ & 0.025 & small-population facilitation rate \\
 \hline
 $f_\max$ & 2 & maximum benefit of facilitation\\ [1ex] 
 \hline
\end{tabular}
\caption{Values of parameters used in numerics in Section \ref{s:far_from_onset}}\label{fig:par}
\end{figure}

\subsection{Linearization at Uniform State}

In order to consider pattern formation from the uniform state $u_*$, we will need the linearization of \eqref{e:main} at $u_*$. Let 
\begin{equation}\label{e:tdef}
\begin{split}
    t_0 & = \frac{f_\textrm{max}  u_* }{f_\textrm{max}/f_0 +  u_*  },
    \qquad \qquad t_1  = \frac{f_\textrm{max}^2}{f_0(f_\textrm{max}/f_0 +  u_* )^2}, \\
    t_2 & = \frac{-f_\textrm{max}^2}{f_0(f_\textrm{max}/f_0 + u_*)^3},
    \qquad t_3  = \frac{f_\textrm{max}^2}{f_0(f_\textrm{max}/f_0 + u_*)^4}\\
    \end{split}
\end{equation}
represent the first four Taylor coefficients of the saturating facilitation function $\frac{f_\textrm{max}   u}{f_\textrm{max}/f_0 +   u}$, expanded at $u = u_*$. Then consider a solution $u = u_* + \wt u$.
Near the uniform state $u_*$, one finds $\frac{\partial \wt u}{\partial t} = \L (\wt u)$ at the linear level, with 
\begin{equation}\label{e:lin}
\L(\wt u) = k_d\frac{\partial^2\wt u}{\partial x ^2} + gu_*t_1(1-u_*/M)\K_f*\wt u - \frac{gu_*}{M}(1-c+t_0)\K_c*\wt u.
\end{equation}

Note that in Fourier space, the linear operator becomes the multiplication operator
\begin{equation}
    \begin{split}
        \wh{\L(u)}(\omega) &= \left(g u_*\left( t_1 \left(1-\frac{u_*}{M}\right)\wh\K_f(\omega) - \frac{(1-c+t_0)}{M}\wh\K_c(\omega)\right) - k_d\omega^2\right)\wh{{u}}(\omega)\\
    &:= \wh\L(\omega)\wh{{u}}(\omega)
    \end{split}
\end{equation}
where $\wh\K_{f,c}(\omega):= \int_{\mathbb{R}}\K_{f,c}(x) e^{-i\omega x} \, \mathrm{d}x$ %
represent the Fourier transforms of the respective interaction kernels. 

\section{Onset of Patterns: Bifurcation from Uniform State} \label{s:CM}

In this section, we present our analytical characterization of spatial pattern formation in the nonlocal PDE \eqref{e:main} starting near the uniform equilibrium state $u_*$ as well as the stability of emerging patterns near the bifurcation threshold. We present the linear stability analysis of our model in Section \ref{sec:LSA}, exploring in Propositions 1-4 how increasing the death rate $d$ or length-scale $\sigma_c$ of competition can produce pattern-formation for both Gaussian-like and hat-like kernels. We then explore the coperiodic stability of patterns near onset through a center manifold reduction in Section \ref{sec:coperiodicstability}, resulting in our classification of the conditions guaranteeing a subcritical or supercritical pitchfork bifurcation in Theorem \ref{t:main}. Finally, we present the proof of Theorem \ref{t:main} in Section \ref{sec:centermanifoldproof} and provide illustrations of pattern-forming bifurcation for our family of example kernels in Section \ref{sec:stabilityresults}. 

\subsection{Linear Stability Analysis}
\label{sec:LSA}

We start by examining conditions under which the uniform state can undergo a Turing instability. Plugging in a perturbation of the form $e^{\lambda t - i\omega x}$ to the linearized equation $u_t = \L u$, with $\L$ the linearization from \eqref{e:lin}, one obtains the following dispersion relation for $\lambda(\omega):$
\begin{equation}\label{e:disp}
\begin{split}
    \lambda(\omega) &= g u_*\left( t_1 \left(1-\frac{u_*}{M}\right)\wh\K_f(\omega) - \frac{(1-c+t_0)}{M}\wh\K_c(\omega)\right) - k_d\omega^2\\
    &= \wh\L(\omega).
    \end{split}
\end{equation}
 Whether $\lambda$ is positive or negative determines whether a perturbation from the constant state $u_*$ of the form $\cos(\omega x)$ is expected to grow or decay exponentially in time. If $\lambda$ is negative for all frequencies $\omega$, then the uniform state is (spectrally) stable. If $\lambda$ is positive for some nonzero frequency $\omega$, then a Turing instability is possible. 

 We are particularly interested in cases where the uniform state $u_*$ becomes unstable to patterns as a parameter is modulated. We focus especially on the parameters $d$ and $\sigma_c$, which represent the linear death rate and the ratio of spatial scales of competition and facilitation, respectively. Broadly speaking, one would expect patterns as $d$ and $\sigma_c$ increase, since increasing $d$ represents environmental harshness, and increasing $\sigma_c$ corresponds to competition (inhibition) occurring over a larger spatial scale than facilitation (activation), a mechanism paramount to Turing instabilities. 

 We find that indeed, $d$ and $\sigma_c$ trigger modulational instabilities for Gaussian-like kernels, provided the diffusion is sufficiently small.  We state this formally in Propositions \ref{p:gauss_suff1}-\ref{p:gauss_suff2} below.  %
 
 \paragraph{Bifurcations for Gaussian-like kernels.}
 %We now state this formally. 
 In all of the following, the auxiliary parameters $g, f_{\max}, M > 0$ and $0 < c < 1$ are fixed but arbitrary. We require $f_0 > \frac{1-c}{M}$, so that the uniform dynamics are in the tipping case, but otherwise $f_0$ is also arbitrary. In Propositions \ref{p:gauss_suff1} and \ref{p:gauss_suff2}, we consider $\K_c, \K_f$ of Gaussian type, with $\wh{\K_c}, \wh{\K_f}$ are nonnegative functions. We also need one more assumption about the kernels: either $\int_\R x^2 \K_f, \int_\R x^2 \K_c < \infty$, or $\K_c$ and $\K_f$ are scaled versions of the same kernel, with $\frac{1}{\sigma_c}\K_c(\frac{\cdot}{\sigma_c}) = \K_f(\cdot)$. In other words, if the kernels are not sufficiently localized, then we assume the facilitation and competition kernels are the same up to scaling by $\sigma_c$. The first alternative of kernel localization is quite mild, and the only kernel among the examples given that does not satisfy it is the algebraic kernel, which decays like $\frac{1}{x^2}$ as $|x| \to \infty$. 
%Given this setup, we have the following: 
Given this setup, we have the following:

\begin{proposition}[Bifurcation in $\sigma_c$ for Gaussian-like kernels]\label{p:gauss_suff1}
   Fix $0 < d < d_\textrm{tip}$.
    \begin{enumerate}
        \item There exists $\sigma_c > 0$ small enough such that $\wh\L_{\sigma_c}(\omega) < 0$ for all $\omega$, and the uniform state is linearly stable. 
        \item If the unscaled kernels satisfy $\wh{\wt{\K}_f}(\omega) \le \wh{\wt{\K}_c}(\omega)$, then the above holds for $\sigma_c = 1$.
         \item For sufficiently small diffusion coefficient $k_d$, there exists $\sigma_* >1$ and $\omega_* > 0$ such that $\wh\L_{\sigma_*}(\omega_*) > 0$. When $\sigma_c = \sigma_*$, the uniform state is linearly unstable to perturbations with frequency $\omega_*$. 
    \end{enumerate}\end{proposition}

In other words, for kernels with nonnegative Fourier transforms, and with small enough diffusion constant, increasing the competitive range $\sigma_c$ will always generate a modulational instability of the uniform state.  The second statement in the proposition adds that if the unscaled facilitation kernel is at least as localized as the unscaled competition kernel, in the sense that $\wh{\wt{\K}_f}(\omega) \le \wh{\wt{\K}_c}(\omega)$, then the modulational instability occurs in the parameter regime $\sigma_c > 1$, where the competitive range is longer than the facilitative range. That is, the fundamental Turing idea of long-range inhibition and short-range activation is not just sufficient, but necessary for patterns when the model uses Gaussian-like kernels. The localization here is framed is in terms of the Fourier transforms, but a sufficient condition is that each of the higher moments $\int_\R x^m\wt\K_f(x)dx$, $m \ge 4$, is less than or equal to that of the competition kernel $\wt\K_c$, which naturally includes the case where the kernels are the same up to scaling. (The integrals and 2nd moments are equal by assumption.) In practice, facilitation mechanisms, such as canopy shading and increase in soil-water absorption, have an inherently finite range, so it is quite reasonable for them to be more localized than competitive interactions.

    \begin{proposition}[Bifurcation in $d$ for Gaussian-like kernels]\label{p:gauss_suff2}
  Fix $\sigma_c > 1$.
    \begin{enumerate}
         \item If $d = 0$, then $\wh\L_{d}(\omega) < 0$ for all $\omega$, and the uniform state is linearly stable. 
        \item  If the diffusion coefficient $k_d$ is sufficiently small, then there exists a death rate $0 < d_* < d_\textrm{tip}$ and $\omega_* > 0$ such that $\wh\L_{d_*}(\omega_*) > 0$. When $d = d_*$, the uniform state is linearly unstable to perturbations with frequency $\omega_*$. 
    \end{enumerate}\end{proposition}

Proposition \ref{p:gauss_suff2} states that for sufficiently small diffusion, the death rate $d$ will always generate a modulational instability, in the tipping case $f_0 > \frac{1-c}{M}$. We highlight this fact because the initial pattern onset is an important feature of vegetative modeling -- environmental harshness should organize pattern formation. We also remark that the nonlocal facilitation term is important here: if $f_\max = 0$, so that the nonlocal facilitation is not present, then the third statement would not hold -- the uniform state would always be linearly stable.

\paragraph{Bifurcations for hat-like kernels.} For hat-like kernels, the story is more complicated, and general statements are more difficult to make.
In fact, for any hat-like kernel, with sufficiently small diffusion, the uniform state is already unstable to patterns at $d = 0$. We state this formally in Propositions~\ref{p:hat_suff1}-\ref{p:hat}, which apply to hat-like kernels for which the Fourier transforms are not assumed to be non-negative.

\begin{proposition}[Bifurcation in $\sigma_c$ for Hat-like kernels]%
\label{p:hat_suff1}
   Fix $0 < d < d_\textrm{tip}$.
    \begin{enumerate}
        \item There exists $\sigma_c > 0$ small enough such that $\wh\L_{\sigma_c}(\omega) < 0$ for all $\omega$, and the uniform state is linearly stable. 
         \item For sufficiently small diffusion coefficient $k_d$, there exists $\sigma_* >1$ and $\omega_* > 0$ such that $\wh\L_{\sigma_*}(\omega_*) > 0$. When $\sigma_c = \sigma_*$, the uniform state is linearly unstable to perturbations with frequency $\omega_*$. 
    \end{enumerate}\end{proposition}

  \begin{proposition}[Bifurcation in $d$ for hat-like kernels]\label{p:hat} Let $\K_c$ be such that $\inf_\omega \wh\K_c(\omega) < 0$. Let $d = 0$. If the diffusion coefficient $k_d$ is sufficiently small, then there exists a nonzero frequency $w_*$ such that the uniform state is unstable to perturbations with frequency $\omega_*$
  \end{proposition}

In other words, for hat-like kernels, the ratio of spatial scales still organizes pattern formation, albeit without the restriction that competition always be longer-range than facilitation -- but the death rate does not in general control the onset of patterns. It is possible for the uniform state to regain and then re-lose stability as $d$ increases, so Proposition \ref{p:hat} does not entirely preclude Turing instabilities caused by increasing $d$. However, for the two example hat-like kernels here, this takes very careful parameter tuning, and requires unphysically large diffusion or small $\sigma_c$. Instead, typically, the uniform state is either stable or unstable for the entire range of $d \in [0, d_\tip)$. In biological terms, Proposition \ref{p:hat} reflects that for sufficiently steep or sudden dropoff of competition with distance (such as for the hat kernel), patterns are preferred even in abundant conditions.

  The proofs of Propositions~\ref{p:gauss_suff1}-\ref{p:hat} are provided in the Appendix. 

We can also visualize how changes in $d$ and $\sigma_c$ allow the emergence of spatial patterns by visualizing the dispersion relation for example kernels and various parameter values. We present the dispersion relation for both a Gaussian kernel and a hat kernel in terms of changing values of the length-scale $\sigma_c$ of competition in Figure \ref{fig:disp_sd_c}, and we present analogous dispersion relations for varying death rates $d$ in Figure \ref{fig:disp_d}. These figures allow us to see that increasing $\sigma_c$, as well as increasing $d$ in the Gaussian-like case, allows for the progressive emergence of unstable wavenumbers. Figure $\ref{fig:disp_sd_c}$ also shows the strong effect of $\sigma_c$ on the most unstable wavenumber. From Figure \ref{fig:disp_d}, we note that for the hat kernel, patterns are already present in the case of zero background death rate $d = 0$. We also observe there that the spatially uniform solution $\omega = 0$ reaches neutral stability as $d$ increases to the tipping point $d_{\tip}$.

   \begin{figure}[!ht]
     \centering
     \includegraphics[width=0.4\linewidth]{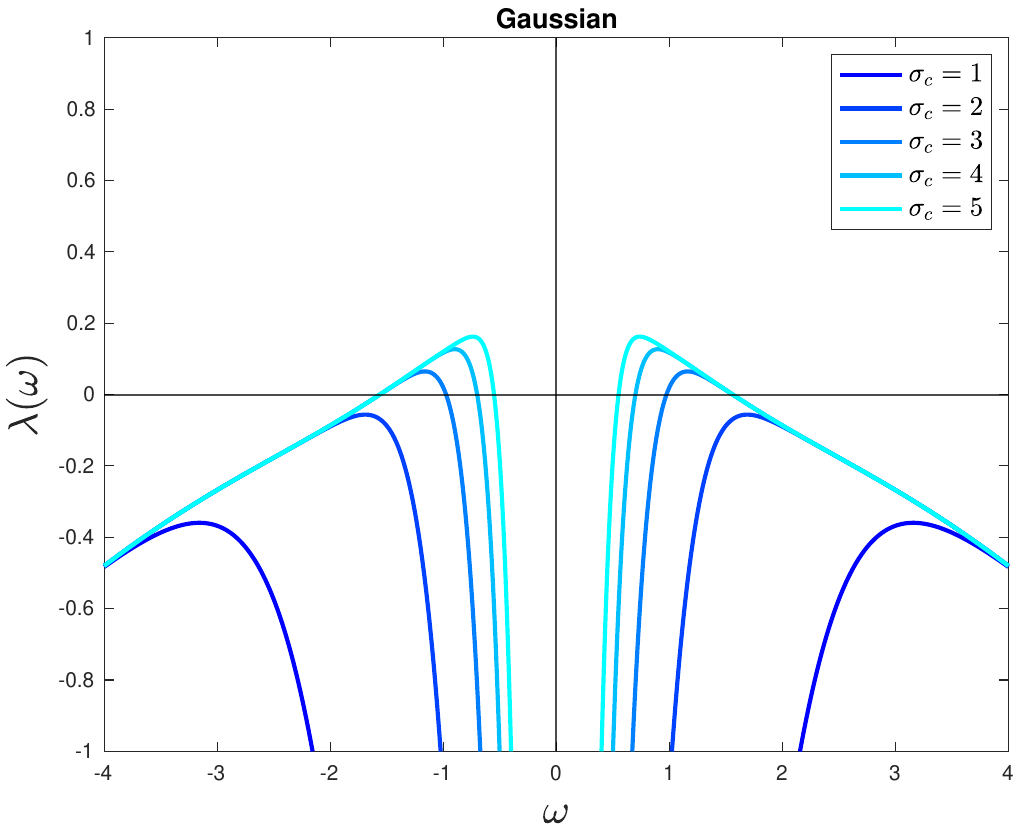}
     \includegraphics[width=0.4\linewidth]{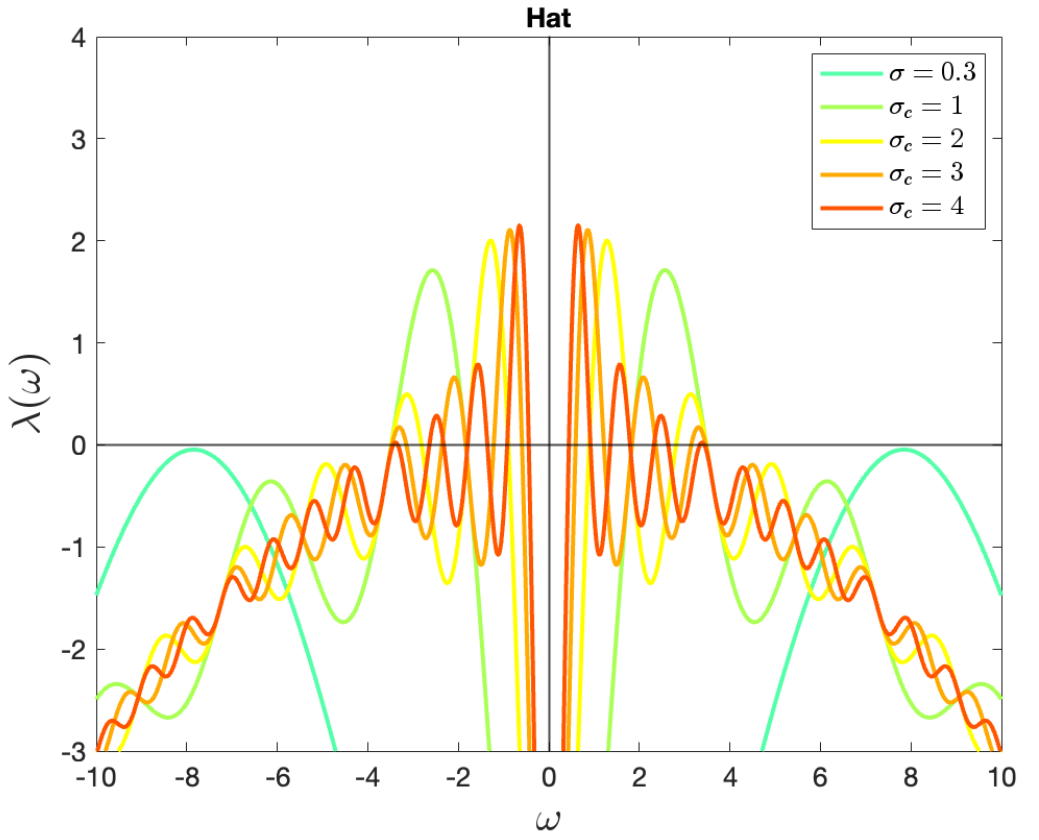}
     \caption{Plots of the dispersion relations for the Gaussian kernel (left) and hat kernel (right) for various values of $\sigma_c$, with all other parameters from Figure \ref{fig:par}.}
     \label{fig:disp_sd_c}
 \end{figure}

   \begin{figure}
     \centering
     \includegraphics[width=0.4\linewidth]{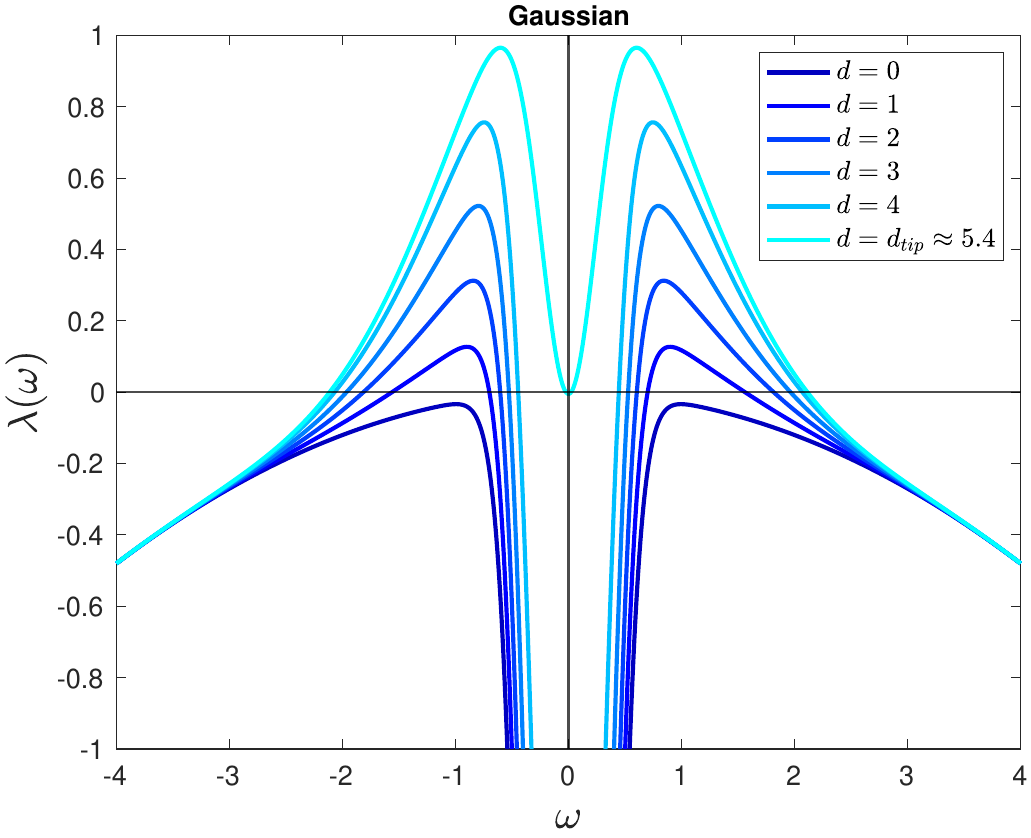}
     \includegraphics[width=0.4\linewidth]{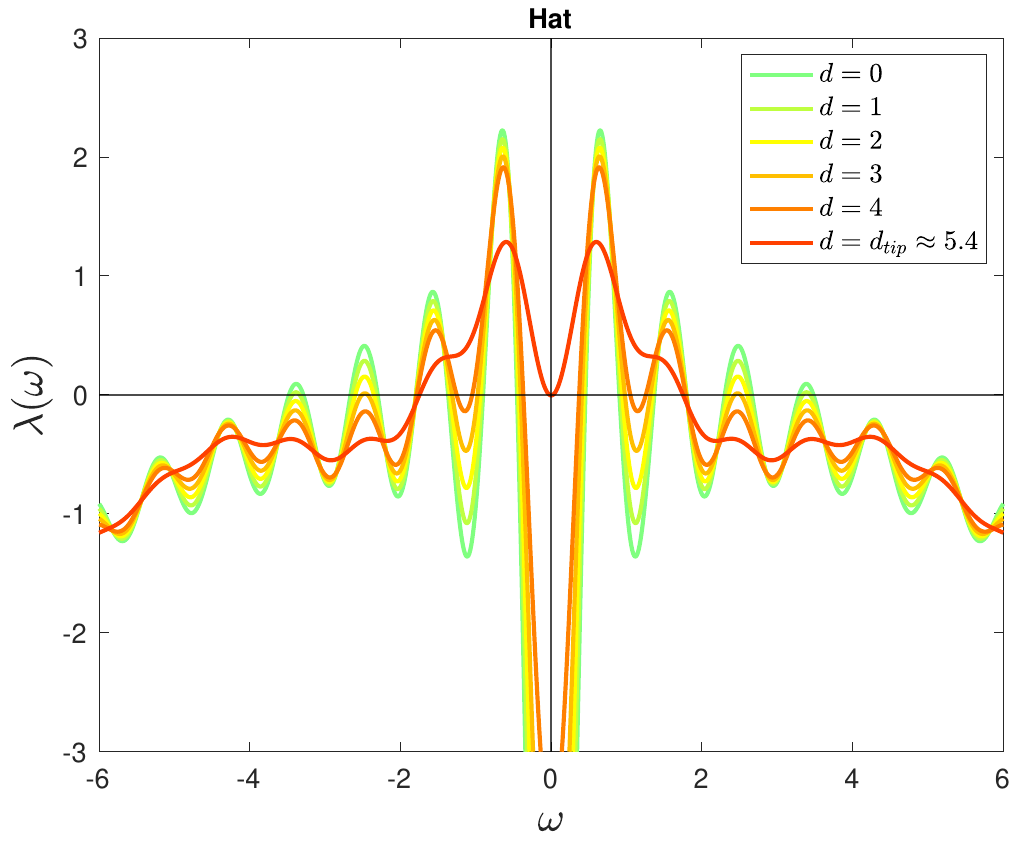}
     \caption{Plots of the dispersion relation for the Gaussian kernel (left) and hat kernel (right) for various values of $d$, with all other parameters from Figure \ref{fig:par}. Note that $\lambda(0)$ is negative except at the tipping point when it equals 0, as expected since the uniform state undergoes a fold bifurcation.}
     \label{fig:disp_d}
 \end{figure}

\subsection{Coperiodic Stability of Bifurcating Patterns}
\label{sec:coperiodicstability}

We next derive conditions for whether pattern-forming bifurcations are super- or sub-critical. In this section, we assume that the linear conditions for a Turing instability are met, and derive information about the resulting patterns near onset. 
Knowing the criticality of the pattern forming bifurcation provides important information about the patterns observed near onset. In the case of a subcritical bifurcation, small-amplitude patterns are unstable and one expects the sudden emergence of large-amplitude patterns. In the supercritical case, small-amplitude, coperiodically stable patterns can be observed. However, the region of small-amplitude patterns may be quite small in parameter space, so that in practice the patterns can quickly become large-amplitude as the parameter increases. 

A center manifold analysis is used to derive a scalar equation that governs the evolution of small-amplitude solutions. The scalar equation has the normal form of a pitchfork bifurcation, and the sign of the cubic coefficient, which we call $A_3$, then determines whether the bifurcation is sub- or super-critical. This method provides rigorous existence and (coperiodic) stability results for patterns of sufficiently small amplitude, and thus determines whether patterns are expected to be small-amplitude deviations of the uniform state, or some other large-amplitude structure.

The section is structured as follows: first, we state hypotheses on the linear equation for pattern formation. Next, assuming these hypotheses, we show that a center manifold exists. Then, by finding an expansion for the manifold, we derive the scalar equation governing small-amplitude solutions to determine the criticality of the bifurcation. After the derivation, we provide examples of how the results translate into observed patterns at onset for various kernels. Throughout, we use $\mu$ as the bifurcation parameter, standing in for either of $\sigma_c$ or $d$, and we write a subscript $\mu$ on the linearization $\L_\mu$ and its Fourier transform $\wh\L_\mu$ to highlight the explicit dependence on the parameter.

\begin{hypothesis}[Assumptions on Linearization at Uniform State]\label{h:linearization}
   Let $\mu \in \{\sigma_c, d\}$, and let $u \equiv u_*$ be a positive spatially uniform solution of \eqref{e:main}. We assume there exists $\mu_* >0$ such that
   \begin{enumerate}
       \item The linearization $\wh \L_\mu$ has $\wh \L_\mu(\omega) < 0$ for all $\omega$ if $\mu < \mu_*$;
       \item For $\mu = \mu_*$, there exists $\omega_* >0$ such that $\wh\L_{\mu_*}(\pm\omega_*) = 0$ and $\frac{d}{d\omega}\wh\L_{\mu_*}(\pm \omega_*) = 0$ (quadratic tangency);
       \item $\wh\L_{\mu_*}(k\omega_*) \neq 0 \quad \forall \ k \in \N$; (non-resonance condition)
       \item $\frac{\partial}{\partial \mu}\wh\L_\mu(\pm\omega_*)\Big|_{\mu = \mu_*} \neq 0$. 
   \end{enumerate}
\end{hypothesis}

We note that Propositions \ref{p:gauss_suff1}-\ref{p:gauss_suff2} above imply that as long as the coefficient of diffusion is sufficiently small, then for any Gaussian-like kernels, all elements of Hypothesis \ref{h:linearization} other than the non-resonance condition hold. The non-resonance condition is a technical condition that restricts patterns on the center manifold to only one wavelength; when it is not satisfied, patterns are still expected; the center manifold is just higher-dimensional and more complicated to analyze. %
We also remark that for any specific choice of kernels and values of parameters, Hypothesis \ref{h:linearization} is easy to verify.

\begin{theorem}\label{t:main}
 Assume Hypothesis \ref{h:linearization}. Then equation \eqref{e:main} undergoes a pattern-forming bifurcation in $\mu$. For $|\mu - \mu_*|$ sufficiently  small, if the coefficient $A_3$ defined in \eqref{e:A3} is negative, the resulting periodic patterns are neutrally stable to co-periodic perturbations. 
\end{theorem}

\subsection{Proof of Theorem \ref{t:main}}\label{ss:CM}
\label{sec:centermanifoldproof}

We now give the proof of Theorem \ref{t:main}. We start by showing existence of a center manifold, then by expanding to find the reduced dynamics.

\paragraph{Justification of center manifold existence}

We consider the equation 

\begin{equation}\label{e:main2}
    \dfrac{\partial u}{\partial t} = g u \left(1-c  + \frac{f_\textrm{max}   \K_f * u}{f_\textrm{max}/f_0 +   \K_{f} * u}\right)\left( 1-\dfrac{\K_{c}*u}{M}\right) - du + k_d \dfrac{\partial^{2} u}{\partial x^2}.
\end{equation}
as an evolution equation on the space $H^2_{\textrm{per},\textrm{even}}([-\frac{L}{2},\frac{L}{2}],\R)$, with $L = \frac{2 \pi}{\omega}$. 

Recall the definitions of $t_0, t_1$ from $\eqref{e:tdef}$. The linearization of \eqref{e:main} at $u \equiv u_*$ is given by 
\begin{equation}
\L_\mu(u) = k_d\frac{\partial^2u}{\partial x ^2} + gu_*t_1(1-u_*/M)\K_f*u - \frac{gu_*}{M}(1-c+t_0)\K_c*u. 
\end{equation}

\sloppy{We have that the negative of $\L_* := \L_{\mu_*}$, considered as a densely defined operator on $L^2_{\textrm{per},\textrm{even}}([-\frac{L}{2},\frac{L}{2}],\R)$, is a sectorial operator, because the negative of $\frac{\partial^2}{\partial x ^2} $ is sectorial and the other terms are compact, hence bounded. In addition, the spectrum of $\L_*$ contains one isolated neutral $0$ eigenvalue and only stable eigenvalues otherwise. Lastly, the nonlinear terms, which are compositions of linear convolution maps with superposition operators, are smooth in a neighborhood of $u \equiv u_*$.  Then, by Theorem 6.2.1 of [Henry, '82], a center manifold exists, tangent to $\ker\L_*$, containing all small bounded solutions to \eqref{e:main} in $H^2_{\textrm{per},\textrm{even}}([-\frac{L}{2},\frac{L}{2}],\R)$. Further, since the nonlinearity is smooth, this manifold can be expanded in a neighborhood of $u \equiv u_*$. We give expansions for this manifold, showing existence and stability of equilibria within $H^2_{\textrm{per},\textrm{even}}([-\frac{L}{2},\frac{L}{2}],\R)$.} 

\paragraph{Center Subspace.}

Recall the dispersion relation $\lambda(\omega)$ from \eqref{e:disp} above:
\begin{equation}
    \lambda(\omega) = g u_*\left( t_1 \left(1-\frac{u_*}{M}\right)\wh\K_f(\omega) - \frac{(1-c+t_0)}{M}\wh\K_c(\omega)\right) - k_d\omega^2.
\end{equation}
 If $\lambda$ is negative for all $\omega$, then the uniform state is (spectrally) stable. By the first assumption in Hypothesis \ref{h:linearization}, when the parameter $\mu$ is less than the critical value $\mu_*$, this is the case. When $\mu = \mu_*$, there is a frequency $\omega_*$ such that $\lambda(\omega_*) = 0$, by the second assumption in Hypothesis \ref{h:linearization}. Then, when $\mu$ increases just past the critical value $\mu_*$, $\lambda(\omega_*)$ becomes positive, by the last assumption in Hypothesis \ref{h:linearization}. In other words, as the parameter $\mu$ increases past the critical value $\mu = \mu_*$, the uniform state becomes unstable to periodic perturbations of frequency $\omega_*$. The non-resonance condition in Hypothesis \ref{h:linearization} guarantees that $\omega_*$ is the only such frequency, within the class of coperiodic perturbations. So one sees that the conditions in Hypothesis \ref{h:linearization} are exactly the linear conditions for a pattern-forming bifurcation. 

When $\mu = \mu_*$, the set of perturbations with $\lambda = 0$ is spanned by $\{\cos(\omega_* x), \sin(\omega_* x) \}$, the kernel of $\L_*$. We reduce this to a one-dimensional set spanned by $\{\cos(\omega_* x)\}$ by restricting to even functions, effectively getting rid of the translation symmetry in the equation. The set $\{a \cos(\omega_*x) \ | \ a \in \R \}$, or the center subspace, parameterizes the steady-state solutions to the linear equation $u_t = \L_* u$ for $\mu = \mu_*$. 

\paragraph{Center Manifold Coefficients.}
Center manifold theory says that for $\mu$ close to $\mu_*$, any small bounded solution of the nonlinear equation can be written as

\[
u_{CM} = a(t) \cos(\omega_*x) + u^\perp(x,a,\widetilde\mu),
\]

where the notation $u^\perp$ refers to the fact that $\langle u^\perp, \cos(\omega_*x)\rangle = 0$ and $\wt{\mu}:= \mu - \mu_*$. Importantly, we have that $u^\perp = \mathcal{O}(a^2)$, and that $u^\perp$ can be expanded in $\wt\mu$ and $a$. This means that the evolution of $u_{CM}$ can be understood entirely by the dynamics of the scalar quantity $a(t)$.
We first use the fact that $u_{CM}$ solves the nonlinear equation to solve for $u^\perp(x,a,\widetilde\mu)$ at the quadratic orders $a^2, a \wt \mu, \wt \mu^2$. 
We write
\begin{equation}\label{e:CMexp}
\begin{split}
    u_{CM} &= a\cos(\omega_*x) + u^\perp(x, a,\wt{\mu}) \\
    &= a\cos(\omega_*x) + u_{20}(x)a^2+ u_{11}(x)a\wt{\mu}+ u_{02}(x)\wt{\mu}^2 + \sum_{j + k \ge3}u_{jk}(x)a^j\wt{\mu}^k,
\end{split}
\end{equation}
where $u^\perp$  is in $ \{\cos(\omega_*x)  \}^\perp$ at every order of $a$ and $\wt \mu$. % and $\wt{\mu} := \mu-\mu_*$. 
Plugging in the expression for $u_{CM}$ into \eqref{e:main}, we obtain 

\begin{equation}\label{e:invCM}
\begin{split}
    &a'(t)\cos(\omega_* x) + 2aa'(t)u_{20}(x)+a'(t)\wt\mu u_{11}(x) + a'(t)(\mathcal{O}(a^2+\mu^2)) = \L_{\mu}(u_{CM}) + N(u_{CM}) \\
    & \quad = (\L_{\mu}-\L_*)(a\cos(\omega_*x)) + \L_*\left(u_{20}(x)a^2+ u_{11}(x)a\wt{\mu}+ u_{02}(x)\wt{\mu}^2\right) + N(a\cos(w_*x)) \\
    & \qquad+ \ \mathcal{O}(a^3, a^2\tilde{\mu}, a\tilde{\mu}^2, \tilde{\mu}^3).
    \end{split}
\end{equation}
We solve for $u_{20}(x),u_{11}(x),u_{02}(x)$ by using invariance; %In order 
to find $u_{20}(x)$, we solve \eqref{e:invCM} %the above 
at order $a^2$. Importantly, we know that $a'(t)$ must be at least $\mathcal{O}(a^2, a\wt\mu)$, because on the linear level when $\mu = \mu_*$, we have $a' = 0$. Then each term on the left hand side of \eqref{e:invCM} is at least cubic order. Inspecting, we find that the only terms of order $a^2$ in \eqref{e:invCM} come from $\L_{*}u_{20}(x)$ and the quadratic terms in the nonlinearity applied to $a\cos(\omega_* x)$. So we need to solve 
\[
-\L_* a^2u_{20}(x) = N(a\cos(w_*x))
\]
at order $a^2$, where $N(\cdot)$ represents the nonlinear terms from \eqref{e:main}. We find $u_{20}(x) = c_0 +c_2\cos(2\omega_*x), $ with 
\begin{align*}
    c_0 &= -\frac{1}{4\wh{\L_*}(0)}N_2(\omega_*), \qquad 
    c_2 = -\frac{1}{4\wh{\L_*}(2\omega_*)}N_2(\omega_*),\\
    N_2(\omega_*)  &= 2g\left(-\frac{1 - c + t_0}{M}\wh{\K_c}(\omega_*) + t_1\left(\left(1 - \frac{u_*}{M}\right) - 
  \frac{ u_*}{M} \wh{\K_c}(\omega_*)\right)\wh{\K_f}(\omega_*)+
   u_*t_2\left(1 - \frac{u_*}{M}\right)\wh{\K_f}(\omega_*)^2\right).
\end{align*}

Next, we solve for $u_{11}(x)$. Plugging $u = a\cos(\omega_*x) + a^2u_{20}(x) + a\wt{\mu}u_{11}(x)$ into \eqref{e:main} results only in $a \wt{\mu}$ terms of the form $a\wt{\mu}\cos(\omega_*x)$, which vanish when projected onto $\{\cos(\omega_*x)  \}^\perp$. Therefore we have that $u_{11}(x) \equiv 0$. Similarly, we find $u_{02}(x) \equiv 0$ since there are no terms in \eqref{e:main} of order $\wt{\mu}^2$ only, both when the parameter $\mu$ represents $\sigma_c$ and when $\mu$ represents $d$.

Thus, we find 
\begin{equation*}
  u^\perp(x,a,\wt{\mu}) = \left( c_0 + c_2 \cos\left( 2 \omega_* x \right) \right)a^2 + \mathcal{O}\left(a^3, a^2 \wt{\mu}, a \wt{\mu}^2, \wt{\mu}^3\right).
\end{equation*}
\paragraph{Reduced Flow on the Center Manifold.}

Having calculated $u^\perp$ to quadratic order, we can find the dynamics for $a(t)$ by inserting $u_{CM}$ into \eqref{e:main}, and projecting the entire equation onto $\cos(\omega_*x)$. After projecting, we obtain
\begin{equation}
a'(t)\cos(\omega_*x) = (\L_{\mu}-\L_*)(a\cos(\omega_*x))  + \mathcal{P} N\left(a\cos(\omega_*x) + a^2(c_0 +c_2\cos(2\omega_*x)) + \mathcal{O}(3)\right),
\end{equation}
where $\mathcal{P}$ is projection onto $\cos(\omega_*x)$. On the right hand side, the first term gives us terms of the form $\cos(\omega_*x)$ when expanding $\L_{\mu}(a\cos(\omega_*x))$ in terms of $\mu$. From the second term, we find terms of the form $\cos(\omega_*x)$ firstly from cubic terms in the nonlinearlity applied to $a\cos(\omega_*x)$, and secondly from quadratic terms applied to the product of $a\cos(\omega_*x)$ and  $c_0 + c_2\cos(2\omega_*x)$. Helpfully, the convolution terms are only multipliers when applied to cosines, so the calculation is not very different than if the nonlinearity were a polynomial.

We find the following reduced dynamics for $a$: 
\begin{equation}\label{e:cm_reduced}
    \frac{da}{dt}= A_1\wt \mu a + A_3 a^3 + \mathcal{O}(a^4).
\end{equation}

We have 
\begin{equation}\label{e:A3}
A_3 = c_0C_2^0 + \frac{1}{2}c_2C_2^2 + \frac{3}{4}C_3,
\end{equation}
where 
\begin{align*}
    C_2^0 &= g\Big(2t_2u_*\left(1 - \frac{u_*}{M}\right)\wh{\K_f}\left(\omega_*\right)
+ t_1\left(1 - \frac{u_*}{M}\right)\left(\wh{\K_f}\left(\omega_*\right) + 1\right)  \\
& \qquad - \frac{t_1u_*}{M}\left(\wh{\K_f}\left(\omega_*\right) + \wh{\K_c}\left(\omega_*\right)\right) 
    - \frac{1 - c + t_0}{M}\left(\wh{\K_c}\left(\omega_*\right) + 1\right) \Big) \\
 C_2^2 &= g\Bigg( 2t_2u_*\left(1 - \frac{u_*}{M}\right)\wh{\K_f}\left(\omega_*\right)\wh{\K_f}\left(2 \omega_*\right)
 +t_1\left(1 - \frac{u_*}{M}\right)\left(\wh{\K_f}\left(\omega_*\right) + \wh{\K_f}\left(2 \omega_*\right)\right)  \\ 
& \qquad - \frac{t_1u_*}{M}\left(\wh{\K_f}\left(\omega_*\right)\wh{\K_c}\left(2 \omega_*\right) + \wh{\K_c}\left(\omega_*\right)\wh{\K_f}\left(2 \omega_*\right)\right) - \frac{1 - c + t_0}{M}\left(\wh{\K_c}\left(\omega_*\right) + \wh{\K_c}\left(2 \omega_*\right)\right) 
\Bigg) \\
 C_3 &= g\Bigg(t_3 u_* \left(1 - \frac{u_*}{M}\right)\wh{\K_f}\left(\omega_*\right)^3 + t_2\left(1 - \frac{u_*}{M}\right)
     \wh{\K_f}\left(\omega_*\right)^2 \\
     &\qquad - \frac{t_2 u_*}{M}\wh{\K_c}\left(\omega_*\right)\wh{\K_f}\left(\omega_*\right)^2 - \frac{t_1}{M}
     \wh{\K_c}(\omega_*)\wh{\K_f}\left(\omega_*\right)\Bigg)
\end{align*}
and 
\begin{equation}\label{e:A1} A_1 = \frac{d(\wh{\L_\mu} (\omega))}{d\mu}\Big|_{\omega = \omega_*,\mu = \mu_*} = \frac{\partial(\wh{\L_\mu} (\omega))}{\partial \mu}\Big|_{\omega = \omega_*,\wt \mu = \mu_*} + \frac{\partial(\wh{\L_\mu} (\omega))}{\partial u_* }\frac{\partial u_*}{\partial \mu}\Big|_{\omega = \omega_*,\mu = \mu_*}. \end{equation}
Note that by the last assumption in Hypothesis \ref{h:linearization}, $A_1 \neq 0$. 
\paragraph{Existence and Coperiodic Stability of Bifurcating Patterns.}
From the reduced equation \eqref{e:cm_reduced}, we obtain a unique steady-state $a = a(\wt\mu)$ for $\wt\mu >0$ sufficiently small, by the implicit function theorem. This equilibrium is stable if $A_3 < 0$. We therefore obtain a bifurcating patterned steady state to the full equation \eqref{e:main}, which is coperiodically stable if $A_3 < 0$.

This concludes the proof of Theorem 1. 

\begin{remark}
    We note that the calculations in this section are essentially identical to those in a formal weakly-nonlinear expansion, an approach with which some readers may be more familiar. In a formal weakly-nonlinear analysis, a small parameter $\varepsilon$ is typically introduced, and both the leading-order perturbation $a\cos(\omega_*x)$ and the bifurcation parameter $\wt \mu$ (standing in for either $\sigma_c-\sigma_*$ or $d-d_*$) are scaled in terms of $\varepsilon$. In particular, this approach typically features a formal asymptotic expansion of the bifurcation parameter and time-scales, which uses the boundary conditions and solvability conditions to deduce the appropriate scaling in terms of $\varepsilon$ and to determine the appropriate amplitude equation for the pattern-forming bifurcation \cite{woolley2022boundary}. Here, our scaling is derived later, from the reduced equation, but it turns out to be identical. The coefficient $a(t)$ in \eqref{e:CMexp} and the bifurcation parameter $\wt \mu$ are treated here as separate small quantities, and their relative scaling emerges from the final reduced equation. 
We find the reduced equation by first writing the solution $u$ as an expansion in $a$ and in the bifurcation parameter $\wt\mu$, then inserting this expanded solution into the differential equation, and lastly projecting onto the subspace spanned by $\cos(\omega_*x)$. This last step can be identified with the `resonance' step in a weakly-nonlinear analysis where only terms of the critical frequency are kept. The reason we use center manifold theory here is that small-amplitude solutions to the reduced equation \eqref{e:cm_reduced} correspond rigorously to solutions of the full equation \eqref{e:main}. 

\end{remark}

\begin{figure}
    \includegraphics[width=0.48\textwidth]{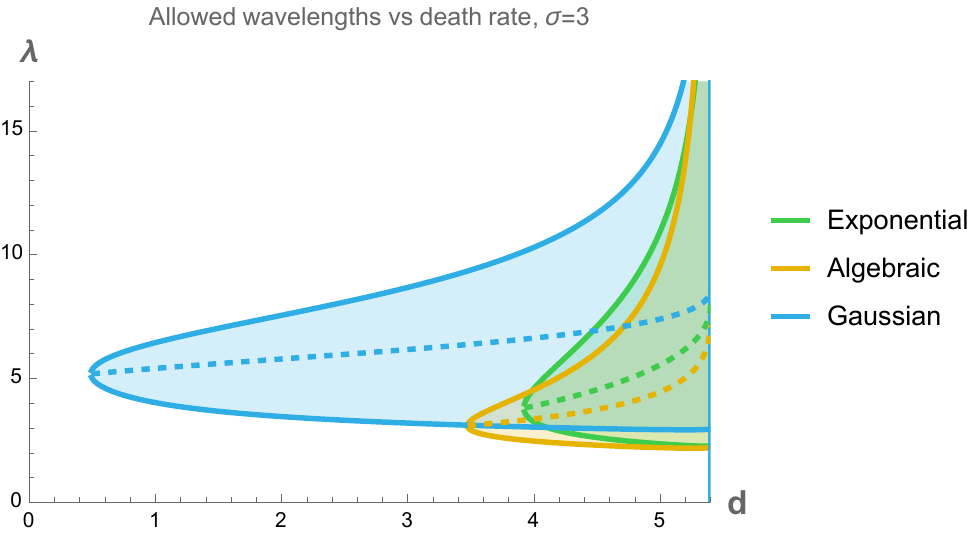}
     \includegraphics[width=0.48\textwidth]{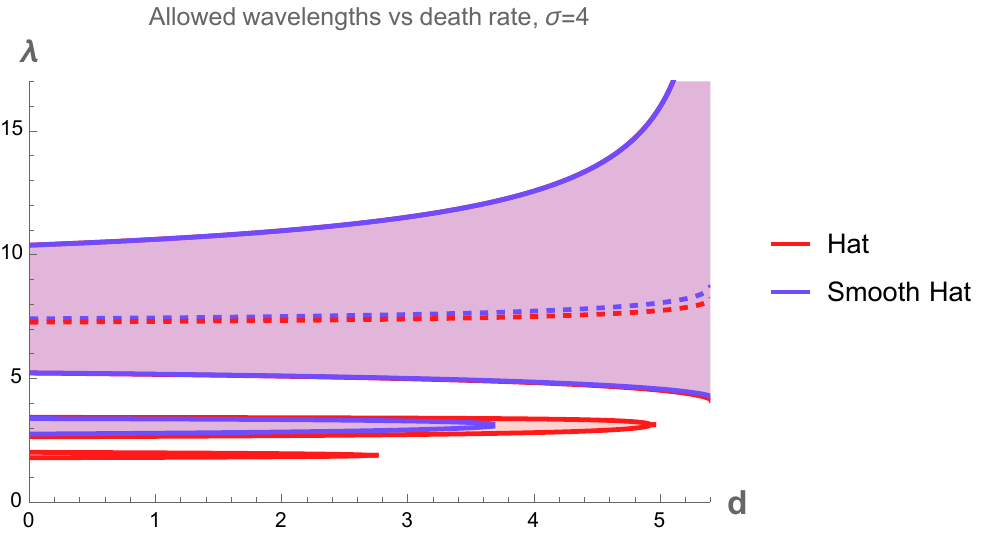}
     \caption{Visualization of wavelengths to which the uniform state is unstable for various $d$. Here, $\sigma_c$ is fixed at $\sigma_c = 3$, and all other parameters are as in Figure \ref{fig:par}. }\label{fig:allowed_wavelengths}
\end{figure}

\subsection{Results}
\label{sec:stabilityresults}

We computed thresholds and wavelengths at pattern onset, as well as coefficients $A_1$ and $A_3$, for the various kernels in Section \ref{ss:kernels}. Pattern onset for the Gaussian-like kernels was computed by fixing $\sigma_c$, then solving for $(d_*,\omega_*)$ such that $\wh{\L}_{d_*}(\omega_*) = \frac{d}{d\omega}\wh{\L}_{d_*}|_{\omega = \omega_*} = 0$, with $\omega_*$ the unique positive maximum of $\wh{\L}_{d_*}(\omega)$. The coefficients $A_1, A_3$ were then computed from \eqref{e:A1},\eqref{e:A3} in the previous section. For the hat and smooth hat kernels, the same strategy was used, but with $\sigma_c$ as used as the bifurcation parameter, for fixed values of $d$. 

Figure~\ref{fig:SDc_vs_d} shows regions of $d$, $\sigma_c$ parameter space that support patterns versus no patterns for the various kernels. Specifically, patterns are expected for parameters $\sigma_c$, $d$ above the corresponding curves. A bifurcation corresponds to the intersection of a vertical line (for a bifurcation in $\sigma_c$) or a horizontal line (for a bifurcation in $d$) with the boundary of the region of patterns. For the exponential, algebraic, and Gaussian kernels, one sees in the figure that increasing either $\sigma_c$ or $d$ can trigger pattern formation. Close to the tipping point, the boundary of the region of patterns approaches $\sigma_c = 1$ for all kernels. We also see from the right panel that in this region where a Turing bifurcation occurs very close to the tipping point, the wavelength $\lambda$ of the bifurcating patterns gets very large ($\omega_* \to 0)$. Examples of unstable wavelengths $\lambda$ as a function of the linear death rate $d$, for a fixed value $\sigma_c = 3$, are provided in Figure~\ref{fig:allowed_wavelengths}. Among Gaussian-like kernels, we find qualitatively similar trends in the bifurcation onset, but pattern onset and wavelength varies significantly depending on the kernel. 

\begin{figure}[!ht]
    \centering
    \includegraphics[width=0.6\linewidth]{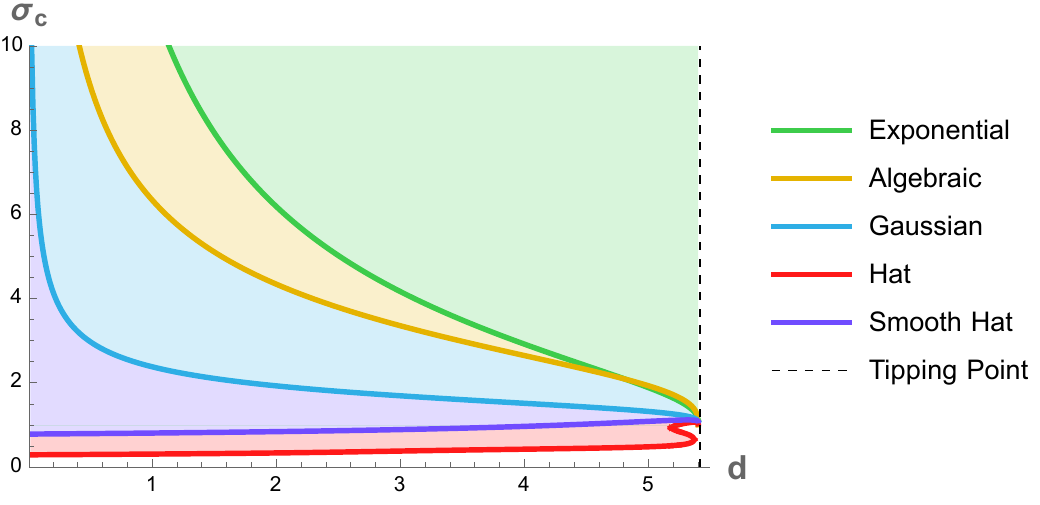}
    \includegraphics[width=0.38\linewidth]{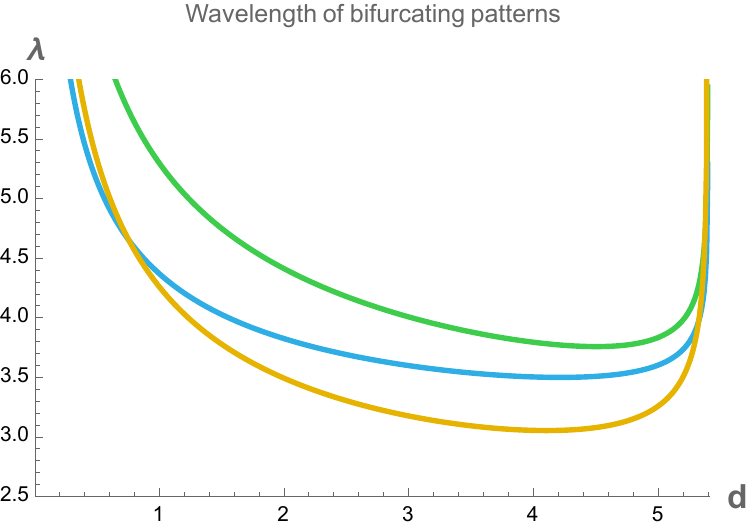}]
  \caption{(Left) Regions of patterns versus no patterns for various kernels in $(\sigma_c,d)$ parameter space. Patterns are expected for each kernel in the region of parameter space above the corresponding curve. (Right) Wavelengths of the bifurcating patterns at each $d$ value for the Gaussian-like example kernels, with $\sigma_c$ at the value on the corresponding curve in the left panel, and all other parameters from Figure \ref{fig:par}.   }
    \label{fig:SDc_vs_d}
\end{figure}

For hat-like kernels, as mentioned above, pattern-forming bifurcations in the parameter $d$ are not guaranteed. Indeed, Figures~\ref{fig:A3} and \ref{fig:allowed_wavelengths} demonstrate that patterns are expected in the case of hat kernels when $d=0$. Another peculiar feature of hat-like kernels is that although bifurcations in $\sigma_c$ are more robust, they can, and often do, happen for ratios of $\frac{\sigma_c}{\sigma_f}$ less than $ 1$. This represents a scenario outside a typical Turing-type regime, where the spatial scale of competition is smaller than the length scale of facilitation. Patterns are still expected as $\sigma_c$ is increased past $\sigma_f$, but appear before the $\frac{\sigma_c}{\sigma_f}$ ratio exceeds 1, different from the Gaussian-like kernels. 

For all the example kernels, we find that pattern-forming bifurcations are supercritical ($A_3 < 0$) except for a very small parameter region near the tipping point; see Figure \ref{fig:A3}. The dashed and solid lines in the left panel indicate supercritical and subcritical bifurcations, respectively. The right panel shows computed values of $A_3$ as a function of $d$, at the bifurcation points $(d,\sigma_c)$ from the curve on the left panel. This figure shows the bifurcation points and associated $A_3$ values with the standard parameters from Figure \ref{fig:par}, however, we found similar results with other values of these parameters, as long as the system remained in the tipping regime. 

\begin{figure}
    \centering
      \includegraphics[width=0.38\linewidth]{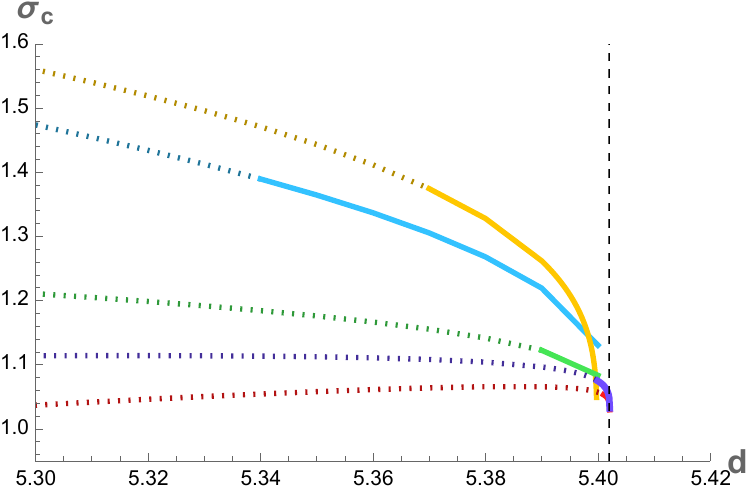} 
    \includegraphics[width=0.6\linewidth]{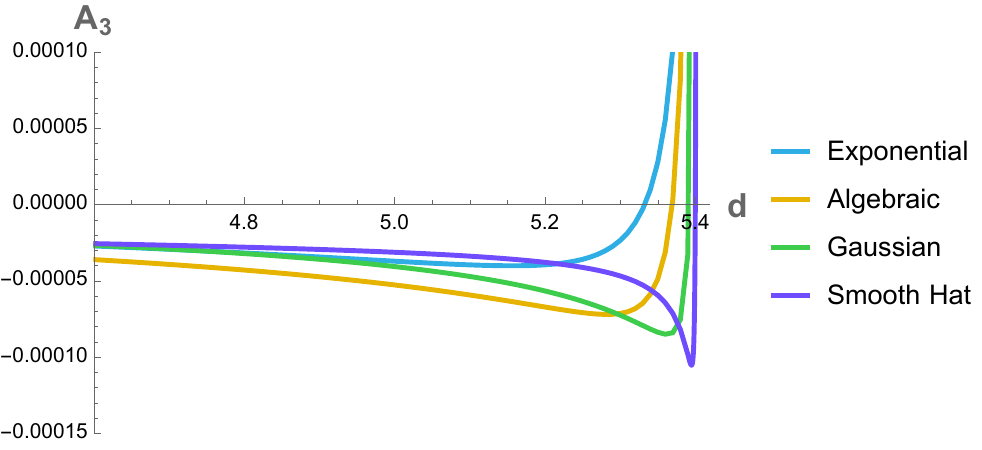}
      \caption{Although bifurcations are supercritical in much of parameter space, they appear to change criticality in a very small parameter region near the tipping point. Shown here is a parameter region near the tipping point where bifurcations are subcritical. All parameters other than $\sigma_c, d$ are as in Figure \ref{fig:par}. (b) coefficients $A_3$. (a) Corresponding bifurcation points, with subcritical bifurcations highlighted as solid lines, and supercritical in dotted lines.}
    \label{fig:A3}
\end{figure}

\section{Patterns far from onset} \label{s:far_from_onset}

Lastly, we consider the selection, stability, and properties of periodic patterns far from pattern onset. As before, we particularly note differences in pattern formation and persistence that arise due to the choice of the competition and facilitation kernels. 

First, we consider periodic patterns that appear on the fixed and finite domain $x \in \left[-\frac{L}{2},\frac{L}{2}\right], \ L <\infty$ with periodic boundary conditions. Following convention, we define the periodic traveling wave coordinate $z = \kappa x - c t$  where $\kappa = 1/L$, $c$ is the temporal frequency or wave speed, and thus $z \in [-1,1]$. Therefore, the desired patterned states are non-trivial, spatially periodic, and stationary in time ($c=0$) solutions of equation~(\ref{e:main}) of the form $u(z,t) = \bar{u}(z)$ where $\bar{u}(z + \Lambda) = \bar{u}(z)$ for some spatial period $\Lambda>0$. This finite domain setting requires that the number of spatial periods $N$ satisfies $\frac{L}{\Lambda} = N\in \mathbb{N}$, which thereby restricts the observed pattern wavenumbers. The spatially homogeneous solutions are trivially spatially periodic, stationary in time solutions of (\ref{e:main}) as well. We numerically compute solutions of~(\ref{e:main}) with root-finding techniques in Matlab and the pattern properties are probed via numerical continuation. A description of the numerical methods used to computed solutions and test their stability are included in the Appendix.

{\bf Patterns far from onset: Gaussian-like kernels.} Properties of the periodic patterns and spatially homogeneous solutions are shown in Figures~\ref{fig:gaussian_max_pattern}-\ref{fig:periodic_L1}. We first describe the results for the Gaussian kernel using the death rate $d$ as the bifurcation parameter, then compare with results obtained for all considered kernels. 
Figure~\ref{fig:gaussian_max_pattern}a depicts the spatially homogeneous and patterned solutions computed on the periodic domain with $L=100$. %
Consistent with the above center manifold analysis, spatial patterns appear as the positive spatially homogeneous solution destabilizes when $d$ is increased beyond $d_* \approx 0.225$. The first stable pattern to be observed has the smallest wavelength (highest $N$). Numerical computation confirms that the pattern emerges as a supercritical bifurcation from the homogeneous solution 
(see Figure~\ref{supfig:supercritical_evidence} in the Appendix). Further increases in $d$ lead to the existence of patterns with longer spatial wavelengths and increasing maximum values. We denote by $d_p$ the largest value of $d$ for which a stable patterned state exists. For a large range of $d$, there are multiple stable wavenumbers; Figure~\ref{fig:gaussian_max_pattern}b provides three examples of stable patterns (solid curves) and one unstable pattern (dashed curve) at $d=5$. Patterns with the longest wavelength remain stable for the highest values of $d$. Numerical continuation of the the unstable patterns suggests that the unstable patterns may originate from  the transcritical bifurcation of the homogeneous state. For small wavenumbers (low $N$), the unstable portion of the continuation curve experiences a saddle node bifurcation and the curve doubles back to higher values of $d$. In this portion of the curve, the unstable patterns are double-peaked, as shown in the dashed unstable periodic pattern in Figure~\ref{fig:gaussian_max_pattern}b.  

\begin{figure}
    \centering
    \includegraphics[width=1\linewidth]{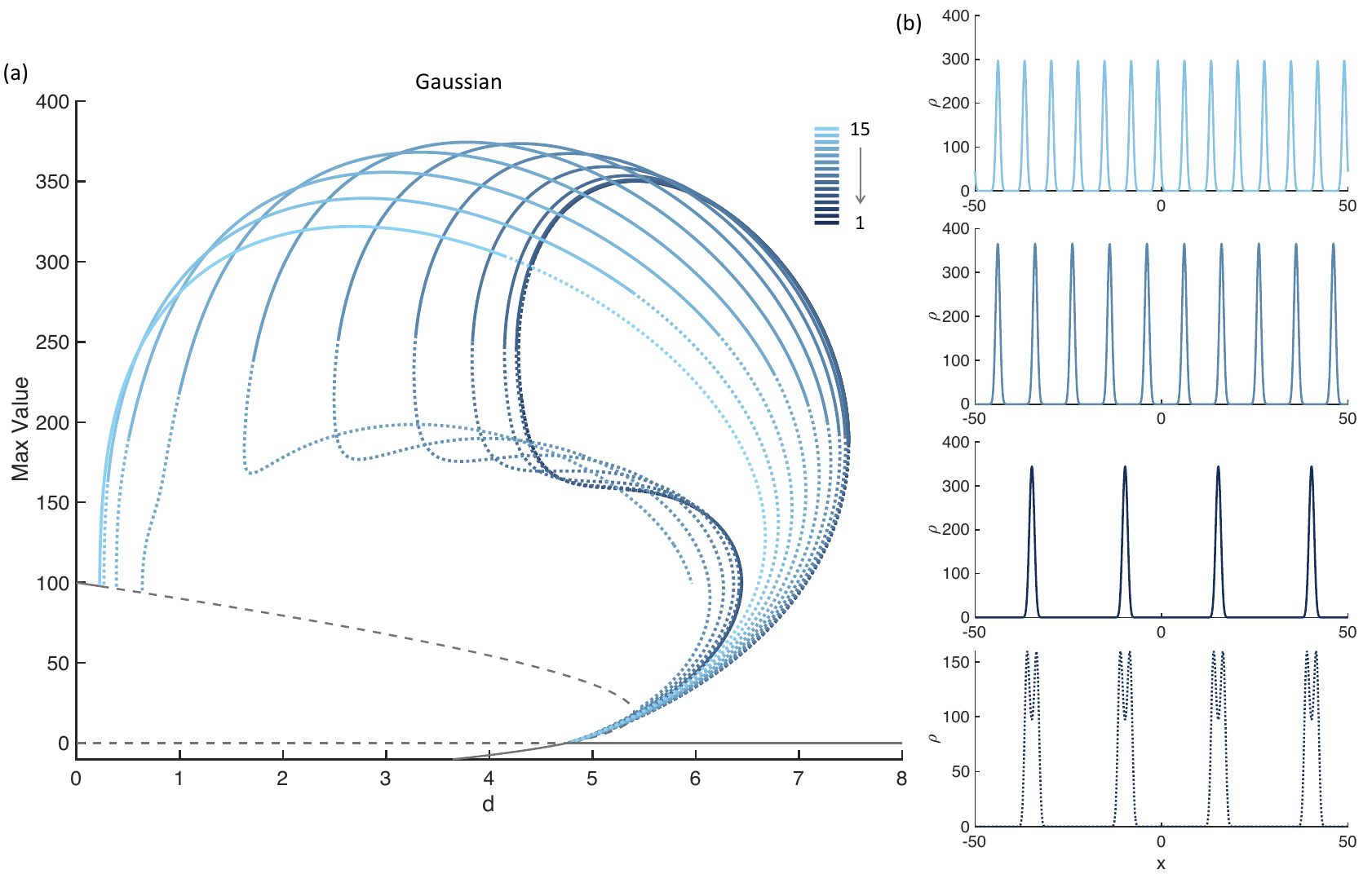}
    \caption{(a) Maximum value ($L^\infty$ norm) of solutions as a function of $d$ for the Gaussian kernel, with patterns computed on the periodic domain with $L=100$. Solid curves indicate stable solutions and dashed are unstable solutions. Patterned solutions are in blue and homogeneous solutions in gray. Legend indicates the number of spatial periods in the periodic pattern. (b) Example stable (solid) and unstable (dashed) periodic patterns at $d=5$.}
    \label{fig:gaussian_max_pattern}
\end{figure}

Periodic patterns states exist and are stable for $d \in (d_*, d_p)$, with $d_p \approx 7.48$ for the conditions shown in Figure~\ref{fig:gaussian_max_pattern}a. For $d \in (d_*, d_\tip)$, the patterned states are the only stable solutions. However, for $d \in (d_\tip, d_p)$ we observe bistability between patterned states and the trivial solution. Hence, this model displays the ``Turing before tipping'' phenomenon that has been investigated in ecological applications (see, for example, \cite{siteur2014beyond,rietkerk2021evasion} and references therein). The persistence of the patterned solutions for $d>d_{T}$ indicates that populations can avoid extinction for high death rates through spatial patterning. 

\begin{figure}
    \centering
    \includegraphics[width=1\linewidth]{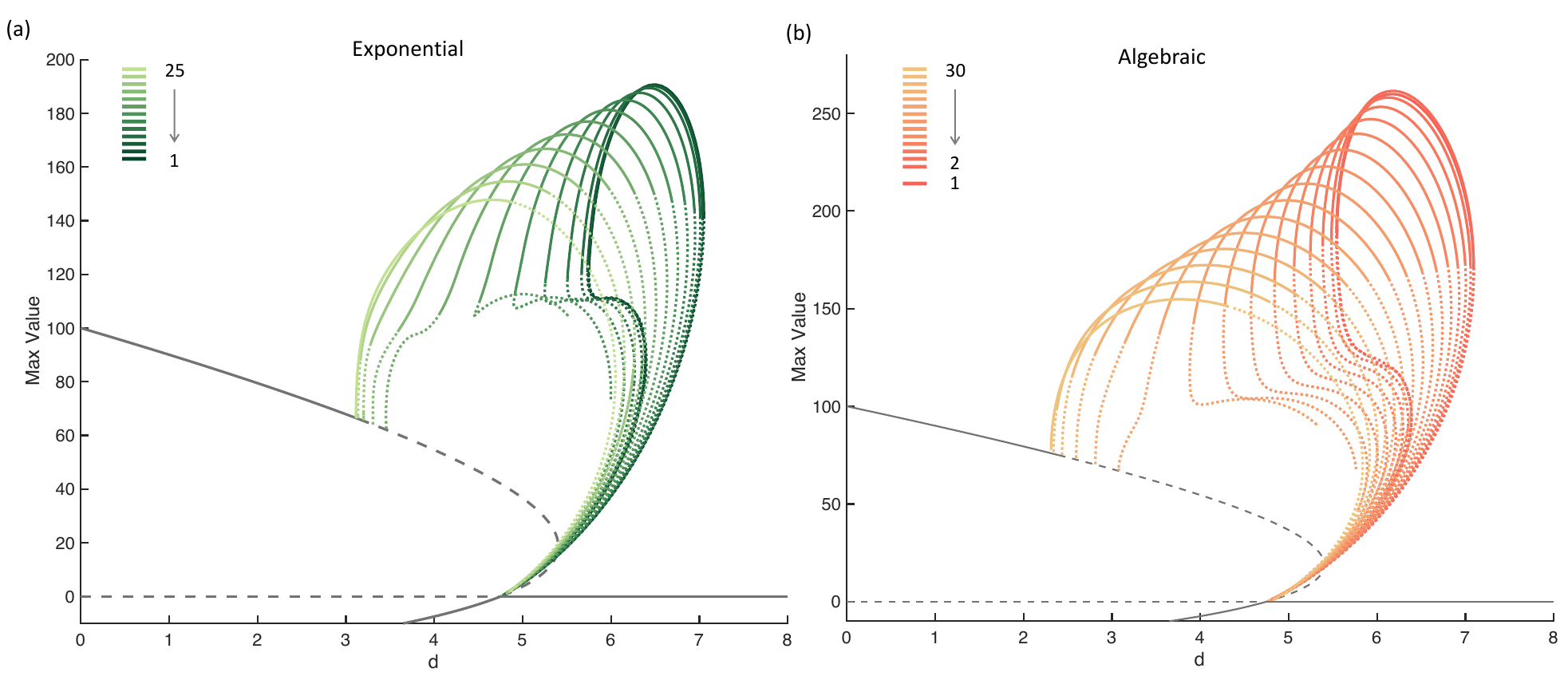}
    \caption{Maximum value of spatially homogeneous equilibrium and periodic patterns as a function of $d$ for the (a) exponential and (b) algebraic kernels. Solid curves indicate stable solutions and dashed are unstable solutions. The legends indicate $N$, the number of spatial periods in the patterned state.  All computations performed on the periodic domain with $L=100$.}
    \label{fig:expo_alg_max_pattern}
\end{figure}

For exponential and algebraic kernels, which fall under the class of ``Gaussian-like'' kernels (see Section \ref{ss:kernels}), the qualitative features of the bifurcation diagram (Figure~\ref{fig:expo_alg_max_pattern}) remain the same. Patterned states emerge upon the destabilization of the positive spatially homogeneous solution, experience multi-stability in the wavenumbers, and coexist with the extinction state. However, the quantitative values of the bifurcation parameters, including $d_*$, and the stable wavenumbers are quite kernel dependent. Patterns first appear at a higher value $d_* \approx 3.11$ for the exponential kernel and have a shorter wavelength (higher $N$) with lower maximum value compared to the Gaussian kernel. Patterned states in the exponential kernel persist until $d_p \approx 7.04$. Similar trends are seen for the algebraic kernel, with $d_* \approx 2.31$ and $d_p \approx 7.09$.

Due to the large range of observed wavenumbers for the exponential and algebraic kernels, the diagrams display a subset of the possible patterns that exist on the finite domain. Curves in the exponential kernel diagram (Figure~\ref{fig:expo_alg_max_pattern}a) represent patterns with odd values of $N\in\{1,\dots,25\}$, and curves in the algebraic diagram (Figure~\ref{fig:expo_alg_max_pattern}(b)) represent values of $N \in \{2,\dots,30\}$ as well as the single wavelength state of $N=1$. In all cases, the pattern that bifurcates from the positive spatially homogeneous solution has the highest value of $N$; patterned solutions with higher values of $N$ are not seen numerically. We note that many of the large wavelength (low $N$) patterned solutions have very similar maximum values and $d$-existence intervals, hence these curves are nearly overlapping in the bifurcation diagrams of Figures~\ref{fig:gaussian_max_pattern}-\ref{fig:expo_alg_max_pattern}.

\begin{figure}
    \centering
\includegraphics[width=1\linewidth]{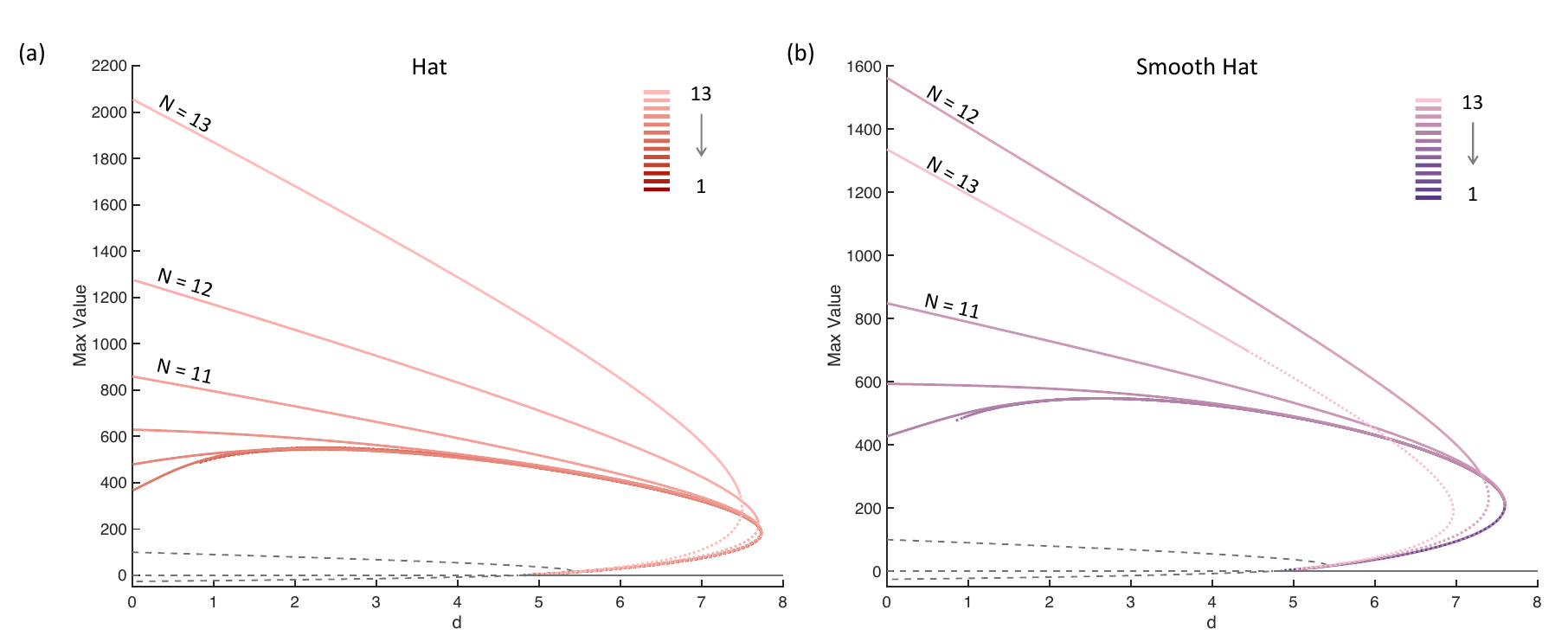}
    \caption{Maximum value of spatially homogeneous equilibrium and periodic patterns as a function of $d$ for the (a) hat and (b) smooth hat kernels. Solid curves indicate stable solutions and dashed are unstable solutions. The legends indicate $N$, the number of spatial periods in the patterned state; labels added to select curves for additional emphasis.  All computations performed on the periodic domain with $L=100$. The curves for $N\leq8$ and largely overlapping in the bifurcation diagrams.  }
    \label{fig:hats_max_pattern}
\end{figure}

{\bf Patterns far from onset: hat-like kernels.} The bifurcation diagrams for patterns computed using hat and smooth hat kernels are shown in Figure~\ref{fig:hats_max_pattern}. Again, the solid curves are stable solutions and dashed are unstable solutions; legends indicate the number of full wavelengths $N$ in the patterned solution. 

For the hat and smooth hat kernels, periodic solutions exhibit qualitatively different existence and stability trends compared to the Gaussian-like kernels. First, for both the hat and smooth hat kernels, multiple periodic solutions exist and are stable for low death rates, including at $d=0$; the positive homogeneous solution is unstable for all positive values of $d$. Hence, the only stable positive solutions are patterned states. Second, we observe that the hat kernel produces patterns with longer wavelength and significantly higher maximum values compared to the Gaussian-like kernels. In figure~\ref{fig:hats_max_pattern}a, we see that the maximum values tend to decrease with the number of spatial periods $N$ for the hat kernel, when $N$ is large. 
%; the labels on the curves are added to help differentiate the colors. 
Patterns for the smooth hat kernel display the same trends (Figure~\ref{fig:hats_max_pattern}b), except for the case of $N=13$, whose maximum value is lower than $N=12$. We note that the maximum values of patterns with wavenumbers N $\lesssim 8$ are largely equal for all values of $d$, resulting in overlapping curves in the bifurcation diagrams in Figure~\ref{fig:hats_max_pattern}. The consistent maximum values in the hat kernels may be due to the abrupt spatial extents of these kernels -- competition between peaks may become negligible once peaks are sufficiently far apart. 

Finally, the large majority of periodic patterns destabilize at the same, high values of $d_p \approx 7.73$ (hat) and $d_p \approx 7.6$ (smooth hat). The location of the instability suggests that patterns in the hat kernels may destabilize for increasing $d$ via a saddle node bifurcation \cite{rademacher2006}. The solutions with $N=12$ and $N=13$ deviate from this trend and these solutions destabilize at lower values of $N$.

\begin{figure}
    \centering
\includegraphics[width=0.75\linewidth]{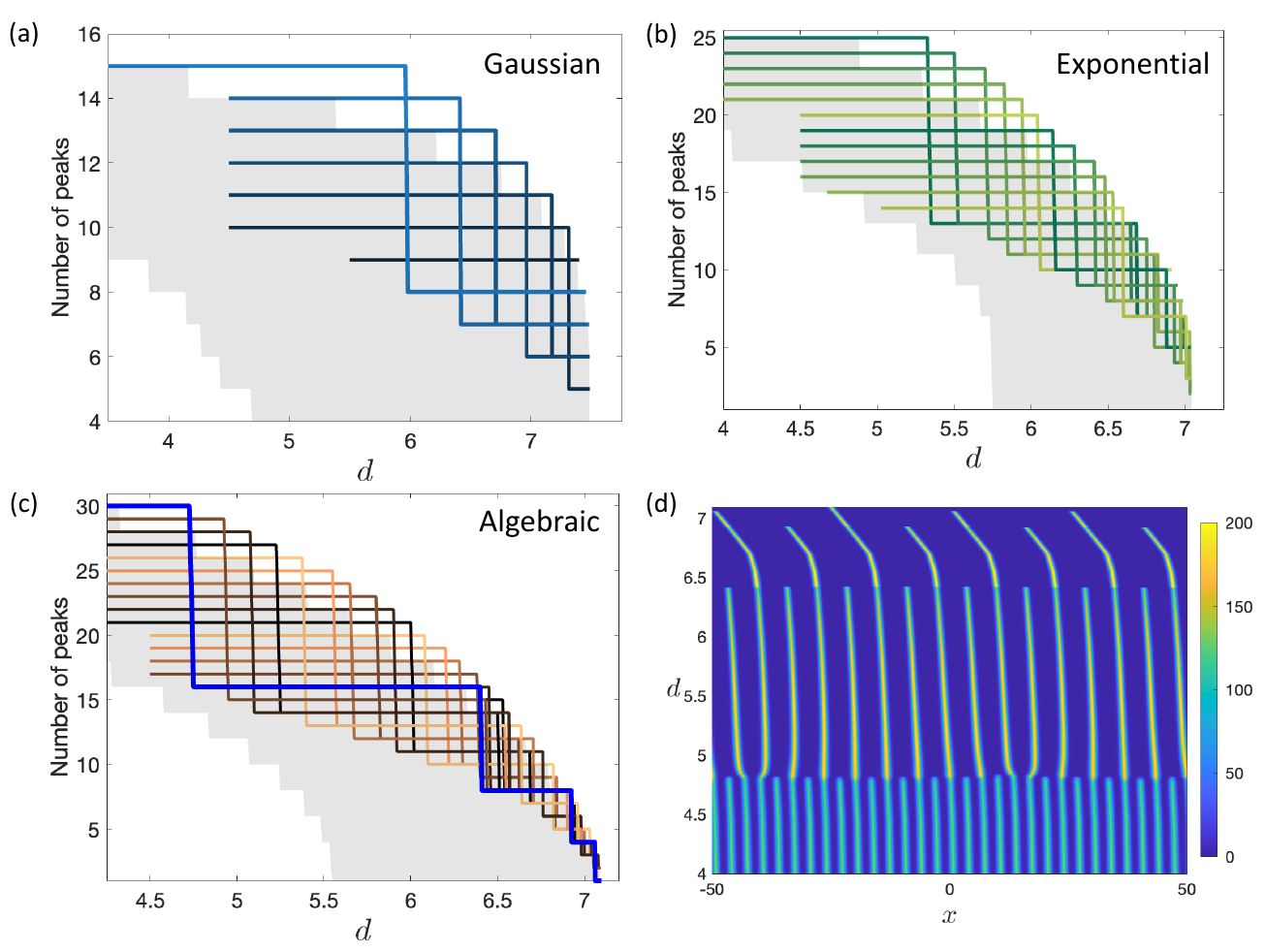}
    \caption{Top row, and bottom left: sample paths through the Busse balloon, where data is from direct simulations with $d$ slowly increasing. Plotted are the number of peaks vs the value of $d$ for the hat and smooth hat kernels. Bottom right: space-time plot of one of the direct simulations, whose path is bolded and colored blue in the bottom left (algebraic kernel) Busse balloon plot. In the spacetime plot, the value of $d$ is plotted in place of the time axis.}
    \label{fig:bouncing_BB}
\end{figure}

{\bf Traveling through the Busse balloon -- slowly increasing the death rate $d$.} 

A noted phenomenon for dryland vegetation patterns is that increased environmental harshness is associated with increased pattern wavelengths \cite{rietkerk2004self,gowda2016assessing}. The stability information in Figures \ref{fig:gaussian_max_pattern} and \ref{fig:expo_alg_max_pattern} aligns with this trend -- for Gaussian-like kernels, the stable patterns increase in wavelength as $d$ is increased. In contrast, the majority of patterns for the hat and smooth hat kernels remain stable until the saddle node is reached. This suggests that wavelength transitions (or lack thereof), which one would expect to see as $d$ is increased slowly, could be quite different between the kernels. To investigate this directly, we include observations from direct simulations where $d$ was gradually increased over time. 

The results are in Figures~\ref{fig:bouncing_BB}a--c and \ref{fig:bouncing_BB_hats} for Gaussian-like and hat-like kernels, respectively. The shaded regions in the figures are coarse approximations of the Busse balloons in each case. Traditionally, Busse balloons depict the stable wavenumbers as a function of some parameter, for patterns posed on the real line, so that the accessible wavenumbers take on all real values in an interval. Here, the shaded region 
%represents the range of the number of spatial periods in the stable patterns as a function of $d$; it 
has been constructed from the numerical continuation results (Figures~\ref{fig:gaussian_max_pattern}--\ref{fig:hats_max_pattern}) and reflects the accessible wavenumbers on a domain with $L=100.$ 
Within Figures \ref{fig:bouncing_BB} and \ref{fig:bouncing_BB_hats}, each curve indicates the path through the Busse balloon taken by a single simulation, for an initial condition yielding a stable number of peaks. 
We see that all patterns become extinct once a high enough value of $d$ is reached -- however, behavior prior to pattern extinction differs for the Gaussian-like and hat-like kernels. 

For the Gaussian-like kernels (Figure~\ref{fig:bouncing_BB}), patterns indeed transition to longer wavelength solutions before extinction. We also see the typical delay phenomenon where peaks disappear slightly after the solution loses stability \cite{asch2025slow}. 
When the wavelength transitions, there appears to be a preference for half of the peaks to disappear, with every other peak lost. The pattern behavior in Figure~\ref{fig:bouncing_BB}d exhibits this period-doubling trend. This selection suggests that the instability is of the spatial period-doubling type, although a deeper investigation is warranted. One possible exception is that patterns with low wavenumbers, $N \lesssim 8$ (Fig. \ref{fig:gaussian_max_pattern}a), $N \lesssim 6$ (Fig. \ref{fig:gaussian_max_pattern}b), $N \lesssim 4$  (Fig. \ref{fig:gaussian_max_pattern}c), tended to be abruptly lost.
It is unclear to us whether this is due to the instability changing type, as can happen near the edge of Busse balloons \cite{doelman2012hopf}, or whether the remaining range of $d$ that supports stable wavelengths is so small that the rate of increase of $d$ would have to be very carefully tuned. 

\begin{figure}
    \centering
    \includegraphics[width=0.75\linewidth]{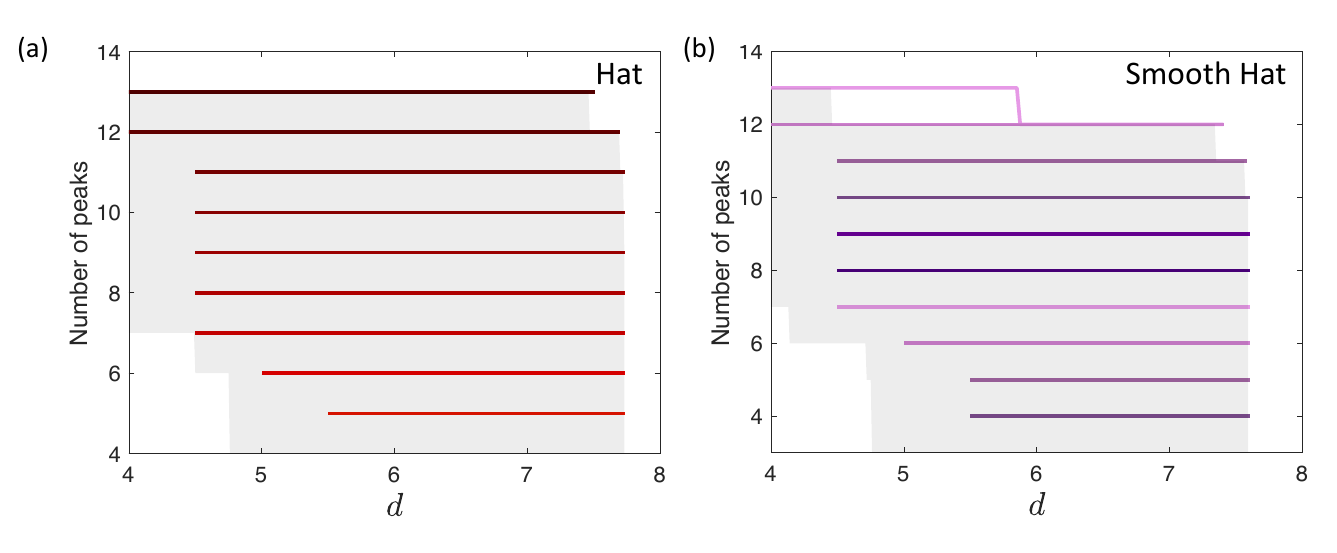}
    \caption{Sample paths through the Busse balloon for the hat and smooth hat kernels, where data is from direct simulations with $d$ slowly increasing. Plotted are the number of peaks as a function of $d$.}
    \label{fig:bouncing_BB_hats}
\end{figure}

For the hat and smooth hat kernels, these wavelength transitions are notably absent. % 
That is, we observe no change in the number of peaks as $d$ increases; the exception is the pattern in the smooth hat that begins with $N=13$ peaks. Hence, in cases that are well approximated by a hat-like kernel, this model predicts there will be no warning signs prior to full extinction.

{\bf Patterns and population productivity.} Finally, Figure \ref{fig:periodic_L1} shows the $L_1$ norms of the homogeneous and patterned solutions.  %
As expected, we find that the $L_1$ norms decrease as conditions harshen (increasing $d$). More interestingly, the patterned states tend to have a higher $L_1$ norm than the spatially homogeneous states -- patterned states support a higher total population than a spatially uniform state would. 
This suggests that patterned states are advantageous for the population, in addition to enabling survival past $d_\tip$. 
This trend is consistent across the tested kernels, but is most apparent for the hat and smooth hat kernels.

\begin{figure}
    \centering
    \includegraphics[width=1\linewidth]{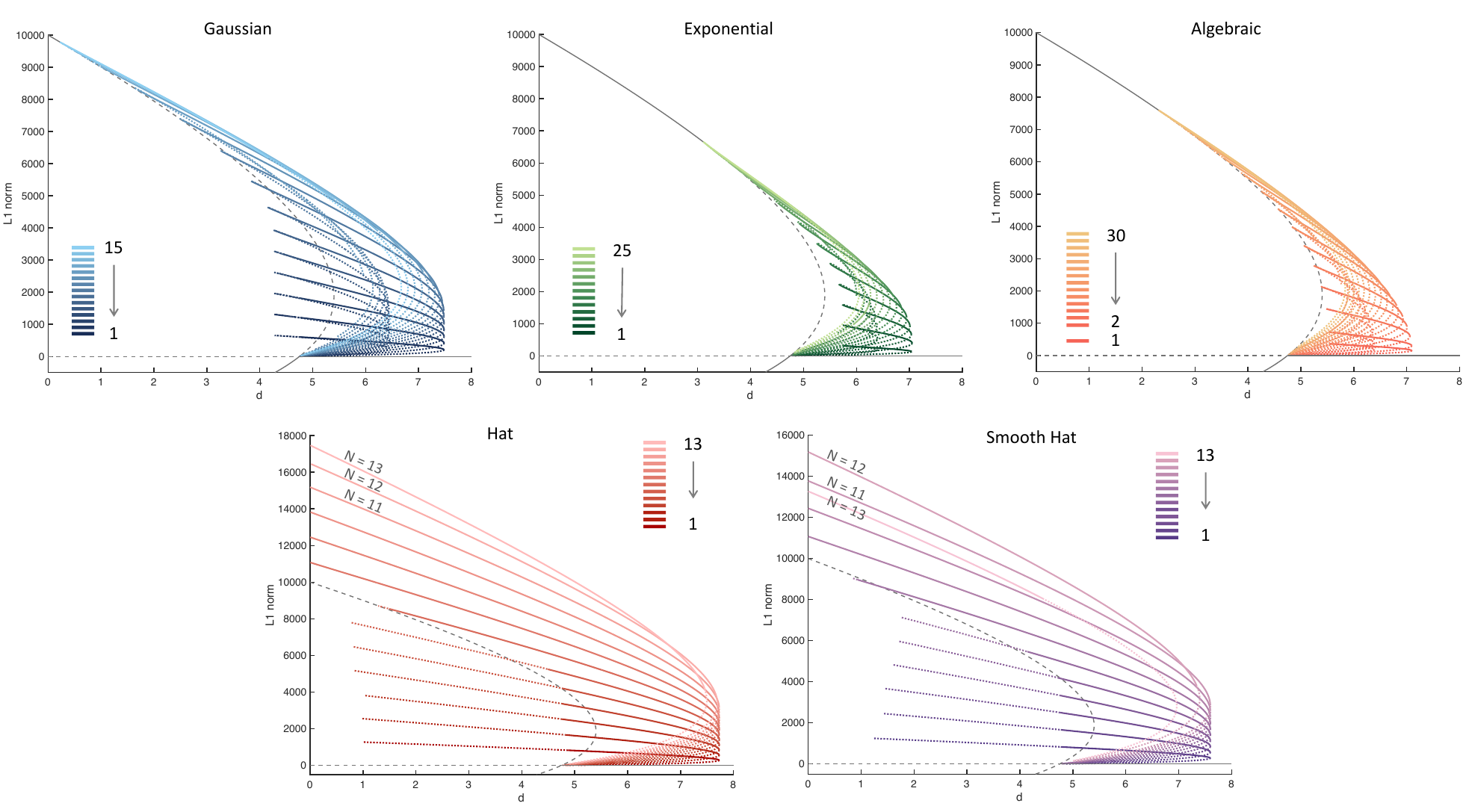}
    \caption{$L_1$-norm of spatially homogeneous equilibrium and periodic patterns as a function of $d$ for Gaussian, exponential,  algebraic, hat, and smooth hat kernels. Solid curves indicate stable solutions and dashed are unstable. Periodic patterns computed on a domain of length $L=100$. The legends indicate the number of spatial periods in the periodic pattern. }
    \label{fig:periodic_L1}
\end{figure}

\section{Discussion}\label{s:discuss}

The above results describe the formation and persistence of patterns in a model with nonlocal competition and facilitation, with emphasis on the role of kernel selection therein. Particular interest was on patterns that bifurcated from a positive uniform stead state upon changes to the relative spatial scale of competition to facilitation and the linear death rate. Our results demonstrate that kernel selection does play an integral role in pattern formation and persistence under worsening conditions. In particular, the same model with various kernels resulted in qualitatively different outcomes for pattern resilience, providing a cautionary example to researchers aiming to make real-world predictions of nonlocal pattern forming systems. 

In the following discussion, we discuss implications of these results, and further directions of study.

\paragraph{Supercritical vs subcritical pattern forming bifurcations.}
Whether pattern-forming bifurcations are super- or subcritical differs across vegetation models in the literature, and there is no clear understand of which model features lead to the different bifurcations. For all parameters explored, we found that pattern-forming bifurcations were supercritical other than a very narrow range of parameter values near the tipping point (Figure~\ref{fig:A3}). Similarly, pattern-forming bifurcations are typically supercritical in Klausmeier-Gray-Scott type models \cite{vanderStelt2013riseandfall}. On the other hand, subcritical bifurcation are found in a long-wavelength, small-amplitude limit of the von Hardenburg model \cite{dawes2016localised}; in  this regime, the model can be reduced to a nonlocal model. Subcritical bifurcations, along with homoclinic snaking, are also seen in a scalar nonlocal model \cite{ruiz2020patterns} for clonal plant growth. Interestingly, in a 3-component reaction-diffusion model for soil water, surface water, and biomass considered by Meron \cite{meron2018patterns}, removing just the positive biomass-water-uptake feedback causes bifurcations to change from supercritical to subcritical. Further investigation into the processes that support sub- versus super-critical bifurcations is warranted, since the distinction can lead to the prediction of physically different phenomena. 

\paragraph{Model formulation}
We have studied a simple but specific scalar nonlocal model. In doing so, a number of choices were made, including but not limited to: the choice to multiply the competition term $(1 - \frac{\K_c * u}{M})$ by the growth term as opposed to subtracting a separate death term; the choice of saturation function for the facilitation; and the fact that dispersal is modeled by diffusion. With the goal of systematic study, it would be interesting to study a more general model, say of the form 
\begin{equation}
    u_t = g(u;d) + s_f(u)\int \K_f(x-x')f(u(x'))dx' - s_c(u)\int \K_c(x-x')c(u(x'))dx' + k_du_{xx},
\end{equation}
where the function $g(u;d)$ encompasses all the local growth and death terms, the functions $s_f(u)$ and $s_c(u)$ represent density-dependent susceptibility to facilitation and competition, respectively, and $f(u)$ and $c(u)$ are density-dependent intensities of facilitation and competition. One might also replace the local diffusion term with nonlocal diffusion, and/or include terms such as $|u_x|^2$ which can arise in modeling clonal growth \cite{ruiz2020patterns}. Such a general analysis would allow direct comparison to models such as \cite{ruiz2020patterns}, in which subcritical bifurcations and snaking were observed, and potentially provide insight into which factors lead to super- or subcritical bifurcations and/or pattern resilience.

\paragraph{Facilitation and resilience.}
We focused mainly on the tipping case, with an interest in exploring the ``Turing before tipping" phenomenon. It is natural to ask what happens in the non-tipping case, where facilitation at low density is small. 
For hat-like kernels, the sub-case of no facilitation ($f_\max = 0$) has been particularly highlighted in the literature, since patterns can be seen for these kernels without facilitative effects \cite{martinez2013vegetation,martinez2014minimal,faye2015modulated,hamel2014nonlocal}. We note that removing facilitation from our model reduces it to a nonlocal Fisher-KPP equation. Notably, the results of Hamel and Rhyzik \cite{hamel2014nonlocal} imply that it is not possible to observe ``Turing before tipping," or more accurately, ``Turing before transcritical," when facilitation is removed from the model studied here. Setting $f_{\max} = 0$, equation \eqref{e:main} can be rescaled into the appropriate form from \cite{hamel2014nonlocal}. %
Then, Lemma 2.2 \cite{hamel2014nonlocal} implies that the two uniform states are the only bounded steady states that exist close to the transcritical bifurcation. Thus, there can be no branch of patterned states that originates before the transcritical bifurcation and continues past it.

One therefore wonders whether patterns emerge and persist beyond bifurcation of the uniform states when the facilitation at low vegetation density is nonzero but small. Although we did not include details of the non-tipping case, preliminary investigation suggests behavior that is similar to the no-facilitation case. The uniform state regains linear stability as $d$ approaches the tipping point, so that likewise, a Turing bifurcation is seen as $d$ is \emph{decreased}, rather than increased. From the kernels and parameters tested so far, this bifurcation also appears supercritical, which would be expected if only the uniform state is stable near the point of desertification. Numerical continuation also gives evidence that patterns do not persist as $d$ is increased past the transcritical bifurcation in the non-tipping case. Combined with prior findings \cite{martinez2013vegetation,martinez2014minimal,faye2015modulated,hamel2014nonlocal}, our results suggest that while facilitation may not be required for pattern onset for hat-like kernels, it may nevertheless play an important role in resilience. Likewise, more investigation is warranted.

\paragraph{Local vs. nonlocal facilitation}
 True tipping in the uniform dynamics requires some type of nonlinear positive feedback, or facilitation process. Here, we have studied the case of nonlocal facilitation, but one could also consider the case of pointwise, or local facilitation. This would correspond in our model to the limit $\sigma_f \to 0$, with $\sigma_c$ fixed. A general scalar model with nonlocal competition and local facilitation was recently considered in \cite{vanderVoort2026vegetation}, where it was shown that local facilitation can indeed lead to pattern formation, even for Gaussian-like competition kernels, which would not yield patterns in the case of only linear growth terms. It would be interesting to see what, if anything, changes qualitatively about pattern onset and pattern resilience, when facilitation is local as opposed to nonlocal.

\paragraph{Patterns in more than one spatial dimension}
We find the expected transition of one-dimensional patterns to higher-wavelengths as the environmental harshness parameter increases \cite{white1970brousse,DHerbs1997FonctionnementEG,vanderStelt2013riseandfall}-- at least in the case of Gaussian-like kernels. A natural question is  whether the sequence of stable patterns in two spatial dimensions also follows the usual motif of gaps giving way to stripes and eventually spots \cite{rietkerk2021evasion}. Such a transition was found in the non-tipping case of a PDE approximation of the Lefevre-Lejeune tiger bush model \cite{lejeune2004vegetation}. It would be interesting to see if this persists in the tipping case and in the full nonlocal model,
%when the Busse balloon extends past the tipping point, 
and especially to see if this transition depends on the choice of nonlocal competition kernel. 

\paragraph{Inferring kernel structure from biological factors.}
 Perhaps the most interesting question of all is whether there is biological data that could determine which spatial kernels are the most accurate. There is some work uses field data of the canopy and root structures of shrubs in southwest Niger to fit facilitative and competitive kernels \cite{Barbier2008}. However, the study only used exponential kernels. Additionally, experimental evidence has been used to determine the strength of inhibition at various distances in marsh tussocks \cite{koppel2006scale}. Another approach for deriving appropriate application specific kernels is via the use of a local model with a fast variable that acts as a resource to mimic the action of a kernel; see, for instance, the inclusion, and then, elimination, of a fast water variable in \cite{martinez2014minimal}. Understanding the dominant feedback mechanisms for dryland vegetation, and translating these to a spatial structure, is a necessary and extremely interesting question, and the authors welcome any biological insights.

\paragraph{Connection to multilevel selection and evolutionary competition.} Hermsen introduced the  model studied here in the context of spatial public goods dynamics, with the goal of understanding multilevel selection for the evolution of altruism \cite{hermsen2022emergent}. 
Mathematical models in that context explore evolutionary competition of group-structured populations in which competition within groups favors cheaters and competition between groups supports groups featuring altruistic individuals \cite{traulsen2006evolution,traulsen2008analytical,luo2014unifying,van2014simple,luo2017scaling}. These models of multilevel selection often presuppose the existence of group-structure in populations, formulating rules for individual-level population dynamics within groups and describing group-level fission, fusion, and extinction events that describe the formation or destruction of groups \cite{simon2010dynamical,simon2016group,simon2024fission,fog2026simulation}. By contrast, Hermsen wanted to use spatial pattern formation as a mechanism of for individuals to self-organize into clusters through a mix of short-range facilitation and long-range competition, with competition within and among clusters arising entirely as emergent behavior from the demographic dynamics of individuals of the population \cite{hermsen2022emergent,doekes2024multiscale}. 
From that perspective, our work contributes to the understanding of how facilitative mechanisms contribute to the patterns themselves and to the emergence of the group-like aggregates that serve as a new unit of selection for evolutionary competition and promoting the evolution of cooperative behavior. 

Our PDE model can also be thought of as a special case of models for the spatiotemporal evolution of altruistic behavior -- here, facilitation (altruism) was constant, but in the original agent-based formulation, altruism was a variable trait. Recent work in the evolutionary game theory literature has also explore nonlocal PDE models featuring game-theoretic interactions according to Hawk-Dove, Stag-Hunt, and Prisoners' Dilemma games mediated by an integral kernel \cite{hwang2013deterministic,aydogmus2017preservation,aydogmus2018discovering}, highlighting the emergence of spatial patterns and invasion fronts that arise due to interaction between competing strategies. This suggests that a natural direction for future work on the Hermsen-type model in the evolutionary games context is to incorporate the existing model of public good dynamics in populations featuring multiple levels of altruism, with the goal of exploring how facilitation, competition, and mutation can interact to promote strategic coexistence, pattern formation, and the persistence of altruism in the presence of cheater strategies. Such extensions of the PDE from the Hermsen model to incorporate multilevel levels of altruism or to consider nonlinear public goods games favoring coexistence of multiple strategies would create substantial opportunities to explore the role of spatial group formation in the emergence of cooperation through evolutionary branching \cite{doebeli2004evolutionary,killingback2010diversity}, pattern formation in evolutionary games \cite{wakano2009spatial,wakano2011pattern,deforest2013spatial,funk2019directed}, and the achievement of cooperation by multilevel selection \cite{boyd1990group,traulsen2006evolution,traulsen2008analytical,luo2014unifying,van2014simple,luo2017scaling}.

\section{Appendix} 

\subsection{Description of numerical methods.}

\textbf{Time evolution of solutions.} Numerical approximates of the time-dependent solutions $u = u(x,t)$ to equation~(\ref{e:nonlocal_veg}) are computed in MATLAB \cite{matlab}. We make the change of coordinate $y = \kappa x$ for $\kappa = 1/L$ and decompose the periodic spatial domain $\left[-1,1\right]$ into \texttt{n}  grid points with uniform spacing $\delta = 2/\texttt{n}$. Throughout, we set $L=100$. Spatial derivatives are approximated using fourth-order centered finite difference methods and convolutions are computed using Fast Fourier Transforms via the MATLAB \texttt{fft}/\texttt{ifft} functions. Approximations are evolved in time using an implicit-explicit (IMEX) Crank-Nicholson and Adams-Bashforth method \cite{ascher1995} that applies the implicit Crank-Nicholson method to the diffusion term and Adams-Bashforth to the nonlinear terms. To balance efficiency and accuracy, we use a uniform time step of $\delta t = 0.05$ (or less) and a minimum of \texttt{n}=1,000 grid points.

Initial profiles for the time evolution were small random or periodic perturbations $\eta(x)$ of a uniform positive state $u_0$. For values of $d< d_{\text{tip}}$, $u_0$ was taken to be the positive spatially homogeneous state. 

\textbf{Computation of equilibrium solutions.} Equilibrium solutions are numerically computed in the periodic co-moving frame $z = \kappa x - c t$ where $\kappa = 1/L$ and $c$ is the temporal frequency or wave speed. Thus, the desired periodic solutions $\bar{u}$ are roots of the equation
\begin{align} \label{e:matrix_root}
0 &= k_d \kappa^2 \partial_{zz} u + c \partial_{z} u + g u \left(1-c + \frac{f_\textrm{max}  \K_f * u}{f_\textrm{max}/f_0 +   \K_{f} * u}\right)\left( 1-\dfrac{\K_{c}*u}{M}\right) - du, \ \ z \in \left(-1, 1\right)&&\\
&u(-1) = u(1), \ \ u'(-1) = u'(1). \nonumber
\end{align}
Since the desired solutions are stationary in time, we expect that $c$ is zero. We note that the spatially homogeneous solutions satisfy the same equations and properties and are computed using the same methods. The change of variable is applied within the convolution terms; in a slight abuse of notation, we still denote the kernels by $\mathcal{K}_c$ and $\mathcal{K}_f$.

The spatial derivatives $\partial_{z}$ and $\partial_{zz}$ are approximated with fourth-order centered finite difference matrix approximations using \texttt{n} grid points with uniform spacing of $\delta = 2/\texttt{n}$. Convolutions are computed by applying the trapezoidal rule to the integrals. Specifically, we write the convolution operator as a $\texttt{n} \times \texttt{n}$ circulant matrix $M$, whose rows define the (shifted) kernel $\mathcal{K}$ on the grid $z \in [-1,1]$. Then, the convolution $\mathcal{K} * u$ is approximated numerically by the scalar-matrix-vector multiplication $\delta M u$ for grid spacing $\delta$. Periodic boundary conditions are directly incorporated into the finite difference matrices and the circulant structure of the convolution matrix. Solutions to the matrix problem (\ref{e:matrix_root}) are found using the root finding function \texttt{fsolve} in MATLAB. Initial guesses for the root-finder are generated from simulations of the PDE. 

Equation~(\ref{e:matrix_root}) is translationally invariant and thereby supports a family of periodic solutions. The standard traveling wave phase condition \cite{Allgower2003,Krauskopf2007}
\begin{align} \label{e:phase_condition}
\int_{-1}^1 \langle u'_{\text{old}}(z), u(z) - u_{\text{old}}(z) \rangle \, \rmd z
\end{align}
with reference solution $u_{\text{old}}$ is added to select a unique solution. The matrix equations~(\ref{e:matrix_root})--(\ref{e:phase_condition}) give $\texttt{nx}+1$ equations, hence we append the temporal frequency $c$ to give $\texttt{nx}+1$ unknowns. Throughout the computations, we track the computed value of $c$ and find that its values are consistently on the order of $1\texttt{e}-7$ or less.

\textbf{Stability.} To determine the spectral stability of a pattern on the finite and periodic domain $z \in\left[-1,1\right]$, we linearize equation~(\ref{e:matrix_root}) about the periodic wave solution $\bar{u}$ and consider the spectrum of the linear operator
\begin{align*}
    \mathcal{L}_{\textrm{pw}} v =& k_d \kappa^2 \partial_{zz} v + c \partial_{z} v + \tilde{g}_1(\bar{u})v+ \tilde{g}_2(\bar{u})\left[ \mathcal{K}_f * v\right] + \tilde{g}_3(\bar{u}) \left[ \mathcal{K}_c * v \right]
\end{align*}
on $L^2_{\text{per}}\left(-1,1 \right)$ with domain $H^2_{\text{per}}\left(-1,1\right)$. The periodic functions $\tilde{g}_1, \ \tilde{g}_2$, and $\tilde{g}_3$  are defined as
\begin{align*}
   \tilde{g}_1(\bar{u}) &=  g \left( 1 - c + \frac{f_\textrm{max} (\mathcal{K}_f * \bar{u})}{f_\textrm{max}/f_0 +  (\mathcal{K}_f * \bar{u}) } \right)\left( 1 - \frac{\mathcal{K}_c * \bar{u}}{M} \right) - d, \\
   \tilde{g}_2(\bar{u}) &= g \bar{u} \left( 1 - \frac{\mathcal{K}_c * \bar{u}}{M} \right) \left( \frac{  f^2_\textrm{max}/f_0}{\left(f_\textrm{max}/f_0 + ( \mathcal{K}_f * \bar{u})\right)^2} \right), \\
   \tilde{g}_3(\bar{u})  &= -\frac{g}{M} \bar{u} \left( 1 - c + \frac{f_\textrm{max}(\mathcal{K}_f * \bar{u})}{f_\textrm{max}/f_0 +  (\mathcal{K}_a * \bar{u})} \right).
\end{align*}
Again, the spatial derivatives are approximated by fourth-order centered finite difference matrices and matrix approximations are used for the the convolutions. Eigenvalues of the matrix approximation of $\mathcal{L}_{\text{pw}}$ are numerically computed in MATLAB using the $\texttt{eig}$ function. A solution is classified as unstable if there exists an eigenvalue with positive real part. 

\textbf{Numerical continuation.} Secant methods \cite{Allgower2003,Krauskopf2007} are used to numerically continue solutions in parameter $d$. During the continuation, equations~(\ref{e:matrix_root})--(\ref{e:phase_condition}) are repeatedly solved as $d$ is gradually increased or decreased and the $L_1$ and $L_{\infty}$ norms of the solutions are recorded and stability tested via the eigenvalues (Figures~\ref{fig:gaussian_max_pattern}--\ref{fig:periodic_L1}). The solution from the previous value of $d$ is used as the reference solution $u_{\text{old}}$. Using the described method, the number of peaks $N$ in the solution remains constant through the continuation in $d$. Therefore, each solution branch in Figures~\ref{fig:gaussian_max_pattern}--\ref{fig:periodic_L1} represents a continuation from a unique initial guess that is constructed with the desired value of $N$.

\textbf{Finite dimensional Busse Balloons} The shaded regions in Figures~\ref{fig:bouncing_BB}-\ref{fig:bouncing_BB_hats} mark stable solutions on the finite domain with $L=100$, as found by numerical continuation. The highest value of $N$ considered in the continuation corresponds to the patterned state that bifurcates from the spatially homogeneous solution. Therefore, it is possible that there exist stable solutions with higher numbers of peaks than represented in Figures~\ref{fig:bouncing_BB}-\ref{fig:bouncing_BB_hats}.

\subsection{Numerical evidence for a supercritical bifurcation.} We provide numerical evidence that upon increasing $d$ the periodic patterns bifurcate from the positive spatially homogeneous solution via a supercritical bifurcation. Define $\epsilon = (d-d_*)/d_*$ to be the relative bifurcation parameter where $d$ is the bifurcation parameter and $d_*$ is the bifurcation value. Figure~\ref{supfig:supercritical_evidence} shows that the amplitude as a function of the relative bifurcation parameter has slope 1/2 on a log-log axis, hence near pattern onset, the amplitude scales like the square-root of $\epsilon$. The data in Figure~\ref{supfig:supercritical_evidence} is generated from numerical continuation of the patterned state with $N=15$ using the Gaussian kernel and $d_*$ is set to be the value of $d$ where the numerical solution has zero amplitude.

\begin{figure}
    \centering
    \includegraphics[width=0.5\linewidth]{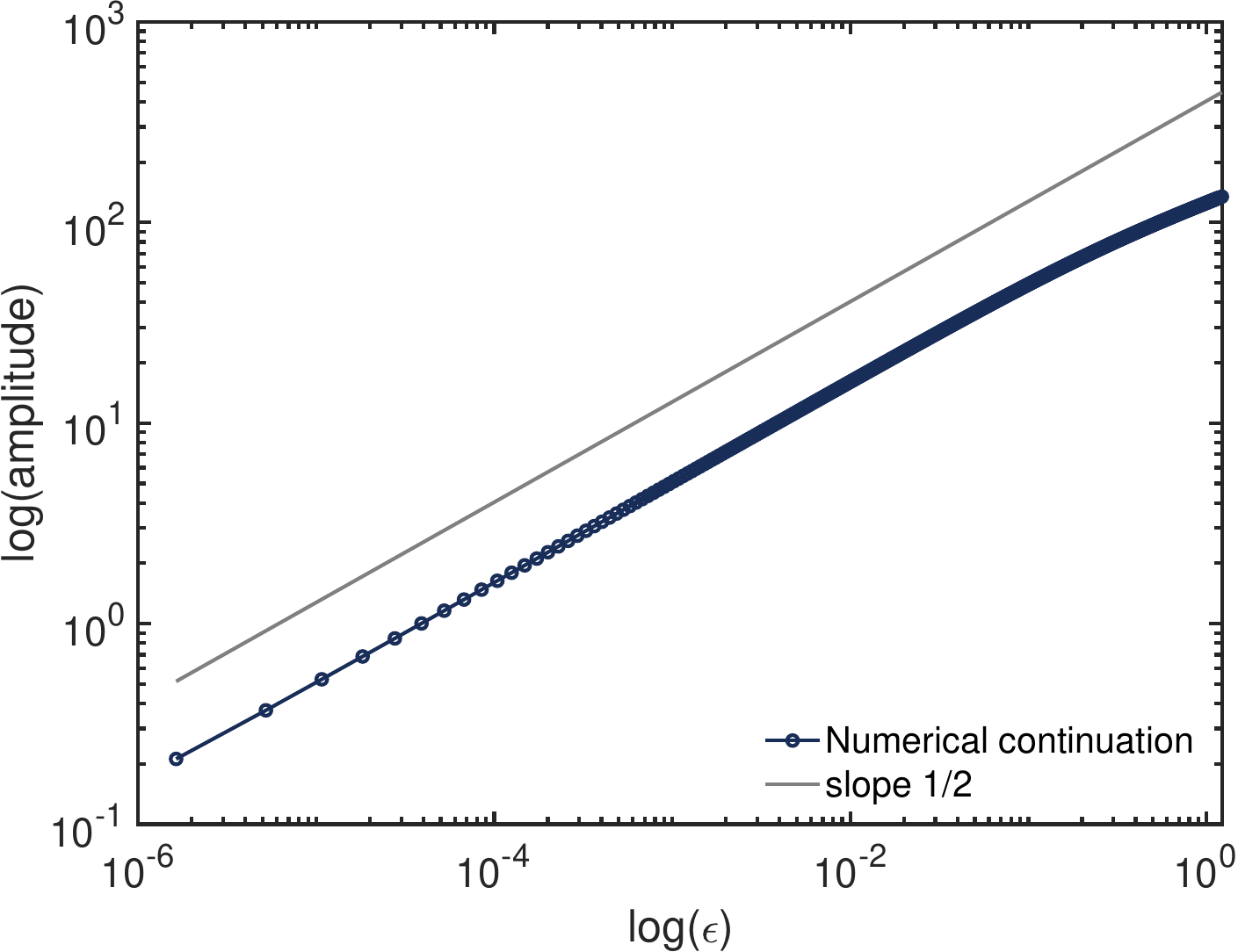}
    \caption{Numerical evidence that the patterned states emerge from the positive homogeneous solution via a supercritical bifurcation. The amplitude as a function of the relative bifurcation parameter $\epsilon = (d - d_*)/d_*$ has slope $1/2$ on a log-log axis. Solution computed using the Gaussian kernel. Pattern had $N=15$ spatial periods.}
    \label{supfig:supercritical_evidence}
\end{figure}

\subsection{Proofs of Propositions \ref{p:gauss_suff1}-\ref{p:hat}}

\begin{proof}[Proof (Proposition \ref{p:gauss_suff1})]
    Fix $0 < d < d_\textrm{tip}$. Consider the linearization $\wh \L$ in Fourier space: 
    \begin{align*}
    \wh\L(\omega) &= \frac{gu_*f_\textrm{max}^2(1-\frac{u_*}{M})}{f_0(f_\textrm{max}/f_0 +  u_* )^2 }\wh\K_f(\omega) - gu_*\left(\frac{1-c}{M}+\frac{f_\textrm{max}  u_* }{M(f_\textrm{max}/f_0 +  u_* ) }\right)\wh\K_c(\omega) -k_d\omega^2\\
    &=: C_f\wh\K_f(\omega) - C_c\wh\K_c(\omega) - k_d\omega^2.
    \end{align*}
    We begin by showing that for $\sigma_c$ sufficiently small, $\wh \L_{\sigma_c}(\omega)<0$ for all $\omega$. Informally, this can be seen because the pointwise limit of $-C_c\wh\K_c(\omega)$ is $-C_c$ for all $\omega$ as $\sigma_c \to 0$. 

Let $\omega_0$ be large enough that $C_f < k_d\omega_0^2$. Then $\wh\L_{\sigma_c}(\omega) < 0$ for $|\omega| \ge \omega_0$, independent of $\sigma_c$. Next, recall that $\sigma_c$ scales $\K_c$ through $\K_c(\cdot) = \frac{1}{\sigma_c}\wt\K_c(\frac{\cdot}{\sigma_c})$, so that the Fourier transform is given by $\wh\K_c(\omega) = \wh{\wt{K}_c}(\sigma_c\omega)$. Then since $\wh{\wt{K}_c}$ is continuous and $C_f - C_c<0$, there exists $\delta > 0$ so that $C_f - C_c\wh{\wt{K}_c}(\omega)<0$ for $|\omega| < \delta$. Taking $\sigma_c$ small enough that $\delta\sigma_c < \omega_0$, we will have 
\[
\wh\L_{\sigma_c} (\omega) \le C_f \wh\K_f (\omega) - C_c\wh\K_c(\omega) \le C_f - C_c\wh{\wt{K}_c}(\sigma_c\omega) < 0, \qquad \omega < \omega_0.
\]
Combining this with the fact that $\wh\L_{\sigma_c}(\omega) < 0$ for $|\omega| \ge \omega_0$, we get that $\wh\L_{\sigma_c}(\omega) < 0$ for all $\omega$. Therefore, if $\sigma_c$ is sufficiently small, the uniform state will be linearly stable. 

For the second statement, note that if $\sigma_c = 1$, then $\K_f = \wt\K_f$ and $\K_c = \wt\K_c$. Then, if $\wh{\wt{\K}_f}(\omega) \le \wh{\wt{\K}_c}(\omega)$, we have

\[
\wh\L_{\sigma_c=1}(\omega) = C_f\wh\K_f(\omega) - C_c\wh\K_c(\omega) - k_d\omega^2 \le (C_f-C_c)\wh\K_c(\omega) - k_d\omega^2 <0.
\]

      Lastly, to see the third statement, we have that the quantity $C_f\wh\K_f(\omega) - k_d\omega^2$ is positive at $\omega = 0$, so there exists $\omega_0 >0$ so that $C_f\wh\K_f(\omega_0) - k_d\omega_0^2 =: \ep > 0$. On the other hand, recall that $\sigma_c$ scales $\K_c$ through $\K_c(\cdot) = \frac{1}{\sigma_c}\wt\K_c(\frac{\cdot}{\sigma_c})$, so that the Fourier transform is given by $\wh\K_c(\omega) = \wh{\wt{K}_c}(\sigma_c\omega)$. Also, $\wh{\wt{K}_c}(\cdot)$ is continuous with limit 0 at infinity, so there exists $\sigma_c$ large enough that $C_c\wh\K_c(\omega_*) = C_c\wh{\wt{K}_c}(\sigma_c\omega_*)< \frac{\ep}{2}$. Then $\wh\L_{\sigma_c}(\omega_*) > 0$, as desired. \end{proof}

    \begin{proof}[Proof (Proposition \ref{p:gauss_suff2})]
   Fix $\sigma_c > 1$. To see the first statement, notice that at $d = 0$, we have that $u_* = M$, so the linearization reduces to 
\[
\wh \L(\omega) = -C_c \wh\K_c(\omega) - k_d\omega^2,
\]
with $C_c > 0$. By assumption, $\wh\K_c(\omega) \ge 0$, with $\wh\K_c(0) = 1$, so $\wh \L(\omega) <0$ for all $\omega$. 
   
   We move on to the second statement. Since $\sigma_c > \sigma_f$, there must be some $\omega_* >0$ for which $ \wh\K_f(\omega_*)  > \wh\K_c(\omega_*)$. If $\K_c, \K_f$ are the same function up to scaling by $\sigma_c$, this is immediate. If they are not, then the assumption of finite 2nd moments guarantees that the kernels can be Taylor expanded to 2nd order in Fourier space:

   \begin{align*}
       \wh\K_f(\omega) &= 1 - \frac{\omega^2}{2} + o(\omega^2), \qquad 
       \wh\K_c(\omega) = 1 - \frac{\sigma_c^2\omega^2}{2} + o(\omega^2).
   \end{align*}
   Since $\sigma_c > 1$, then there is some $\omega$ where $\wh\K_f(\omega) >  \wh\K_c(\omega) $. 
   
   Now, consider again the linearization $\wh \L$ in Fourier space. The linearization depends continuously on $u_*,$ which in turn depends continuously on $d$ for $d \in [0, d_\tip]$. 
   %Since $\sigma_c > \sigma_f$, then there must be some $\omega_* >0$ for which $\wh\K_f(\omega_*)  > \wh\K_c(\omega_*)$.
  When $u_*$ is exactly at the tipping point $ u_\tip = \frac{f_\max}{f_0}\left(\sqrt{\frac{f_\max + f_0 M}{f_\max + 1 - c}}-1\right)$, one can calculate that $C_f = C_c = C >0$. Then we must have 
    \[
      \wh\L(\omega_*) + k_d \omega_*^2 = C(\wh\K_f(\omega_*)-\wh\K_c(\omega_*)) > 0.
    \]
    Since the right hand side depends continuously on $d$, we can then find $d$ sufficiently close to the tipping point so that $\wh\L(\omega_*)+ k_d \omega_*^2 >0$.
   It remains to take $k_d$ small enough, and then $\wh\L(\omega_*) > 0$, as desired. 
\end{proof}

\begin{proof}[Proof (Proposition \ref{p:hat_suff1})] 
The proof of the first statement follows identically to the first statement in Proposition \ref{p:gauss_suff1}, except that the condition on $\omega_0$ changes to $\sup_\omega(-C_c\wh\K_c(\omega)) + \sup_\omega C_f \wh\K_f(\omega) < k_d\omega_0^2$.

The proof of the second statement is identical to the proof of the last statement in Proposition \ref{p:gauss_suff1}. 
\end{proof}

\begin{proof}[Proof (Proposition \ref{p:hat})] At $d = 0$, we have that $u_* = M$, so the linearization reduces to 
\[
\wh \L(\omega) = -C_c \wh\K_c(\omega) - k_d\omega^2,
\]
with $C_c > 0$. By assumption, $\inf_\omega \wh\K_c(\omega) < 0$, and $\wh\K_c(0) = 1$, so there must be some $\omega_*>0$ such that $\wh\K_c(\omega_*)<0.$ Then for $k_d$ small enough, $\wh\L(\omega_*) > 0$, as desired. 
\end{proof}

\section*{Declarations}

\textbf{Competing Interests:} The authors have no competing interests to declare that are relevant to the content of this article

\textbf{Data Availability:} Data sharing is not applicable to this article as no new data were created or analyzed in this study

\bibliographystyle{unsrt}
\bibliography{references}

\end{document}